\documentclass[11pt]{article}    
\usepackage{amsmath, amssymb, amsthm}
\usepackage{fullpage}
\usepackage{verbatim}
\usepackage{graphicx}
\usepackage{subfig, caption}
\usepackage{cancel}
\usepackage{enumerate}
\usepackage{wrapfig}
\usepackage{mathtools}
\usepackage{xcolor}
\usepackage{cleveref}
\usepackage{overpic}
\usepackage{esint}
\usepackage{dsfont}
\usepackage{color}
\usepackage{titling}
\usepackage{soul}
\usepackage[normalem]{ulem}

\makeatletter
\newcommand\ReDeclareMathOperator{%
  \@ifstar{\def\rmo@s{m}\rmo@redeclare}{\def\rmo@s{o}\rmo@redeclare}%
}
\newcommand\rmo@redeclare[2]{%
  \begingroup \escapechar\m@ne\xdef\@gtempa{{\string#1}}\endgroup
  \expandafter\@ifundefined\@gtempa
     {\@latex@error{\noexpand#1undefined}\@ehc}%
     \relax
  \expandafter\rmo@declmathop\rmo@s{#1}{#2}}
\newcommand\rmo@declmathop[3]{%
  \DeclareRobustCommand{#2}{\qopname\newmcodes@#1{#3}}%
}
\@onlypreamble\ReDeclareMathOperator
\makeatother

\newtheorem{theorem}{Theorem}
\newtheorem{lemma}[theorem]{Lemma}

\newtheorem{prop}[theorem]{Proposition}

\newtheorem{proposition}[theorem]{Proposition}

\newtheorem{assumption}[theorem]{Assumption}

\newcommand{\vertiii}[1]{{\left\vert\kern-0.25ex\left\vert\kern-0.25ex\left\vert #1 
    \right\vert\kern-0.25ex\right\vert\kern-0.25ex\right\vert}}
\newcommand{\aut}{{\rm aut}}
\newcommand{\ptw}{{\rm ptw}}

\newcommand{\E}{\mathbb{E}}

\newcommand{\R}{\mathbb{R}}
\newcommand{\Z}{\mathbb{Z}}

\newcommand{\cF}{\mathcal{F}}

\newcommand{\cW}{\mathcal{W}}

\newcommand{\eps}{\varepsilon}

\newcommand{\e}{\epsilon}

\DeclareMathOperator{\supp}{supp}

\DeclareMathOperator{\tod}{\xrightarrow{~d~}}
\DeclareMathOperator{\toas}{\xrightarrow{a.s.}}

\DeclareMathOperator{\Lip}{Lip}

\ReDeclareMathOperator{\Re}{Re}
\ReDeclareMathOperator{\Im}{Im}

\renewcommand{\nabla}{D}

\newcommand{\ochi}{{\overline\chi}}
\newcommand{\Chi}{\mathcal{X}}

\numberwithin{equation}{section}
\numberwithin{theorem}{section}
\def\beq{\begin{equation}}
\def\eeq{\end{equation}}

\Crefname{assumption}{Assumption}{Assumptions}
\Crefname{theorem}{Theorem}{Theorems}
\Crefname{lem}{Lemma}{Lemmas}
\Crefname{cor}{Corollary}{Corollaries}
\Crefname{prop}{Proposition}{Propositions}
\Crefname{theorem}{Theorem}{Theorems}
\Crefname{conjecture}{Conjecture}{Conjectures}

{\setlength{\parindent}{0cm}

\begin{document}

\title{Brownian fluctuations of flame fronts with small random advection}

\setlength\thanksmarkwidth{.5em}
\setlength\thanksmargin{-\thanksmarkwidth}

\author{
	Christopher Henderson\thanks{Corresponding author, Department of Mathematics, The University of Chicago, 5734 S.~University Avenue, Chicago, IL 60637, E-mail: \texttt{henderson@math.uchicago.edu}} 
\ and	
	Panagiotis E.~Souganidis\thanks{Department of Mathematics, The University of Chicago, 5734 S.~University Avenue, Chicago, IL 60637, E-mail: \texttt{souganidis@math.uchicago.edu}}
}

\maketitle
\begin{abstract}
\noindent We study the effect of small random advection in two models in turbulent combustion.  Assuming that the velocity field decorrelates sufficiently fast, we (i) identify the order of the fluctuations of the front with respect to the size of the advection, and (ii) characterize them by the solution of a Hamilton-Jacobi equation forced by white noise.  In the simplest case, the result yields, for both models, a front with Brownian fluctuations of the same scale as the size of the advection.  That the fluctuations are the same for both models is somewhat surprising, in view of known differences between the two models.
\end{abstract}
%


\section{Introduction}

We are interested in the rigorous understanding of the effect of a small random advective term, which varies on large scales, on the asymptotic behavior of two types of fronts arising in turbulent combustion, population dynamics, and various other physical systems, which in the absence of advection yield the same front.


\medskip

The first model is the so-called G-equation.  It is a positively homogeneous of degree one Hamilton-Jacobi equation used to describe front propagation governed by Huygen's principle.  In its simplest form, that is without advection, the G-equation yields fronts moving with constant normal velocity. 
The G-equation 
is derived as a simplified model when the advection varies on an integral length scale.  

\medskip

The second model is an eikonal equation that is related to a turbulent reaction-diffusion equation. 
The combined effects of reaction, advection, and diffusion yield complex behavior, including the failure of Huygen's principle, that has drawn significant mathematical interest.

\medskip

There is a long history of developing and using simplified models for turbulent combustion; we refer the reader to the book of Williams~\cite{Williams}, the introduction of the work by Majda and Souganidis~\cite{MajdaSouganidis}, and references therein.  In~\cite{MajdaSouganidis}, the authors develop a mathematically rigorous framework to understand the connection between the advective reaction-diffusion models and the G-equation.  One of the conclusions is that, when the advection varies on large length scales, the front asymptotics may be different, see~\cite[Appendix~B]{MajdaSouganidis}.

\medskip

In~\cite{MayoKerstein}, Mayo and Kerstein study small advection perturbations of the G-equation and formally obtain that the correction of the front location is given by a Hamilton-Jacobi equation forced by one-dimensional (in the direction of the front) white noise.

\medskip

Here, we provide a rigorous mathematical justification of this result.  
 In addition, we study the asymptotics of the second model, that is, the eikonal equation.

\medskip

A somewhat surprising conclusion is that these two models have the same highest-order asymptotics and first-order correction.  In particular, the result implies that the disparity found in~\cite{MajdaSouganidis} is a large-advection phenomenon.

\medskip

We next describe the setting. 
We work in $\R^n$ with $n\geq 2$ and denote elements as $(x,y)$ with $x \in \R^{n-1}$ and $y\in \R$.  We also write $(x,\xi)$ for elements of $\R^n$ with $x\in \R^{n-1}$ and $\xi \in \R$, when $\xi$ plays the role of a ``slow variable.''  Finally, we set $\R_\pm := \{y \in \R: 0 < \pm y < \infty\}$.

\medskip

For our results, we require an appropriate smooth approximation of white noise, often referred to as mild white noise, which we denote by $w$.  The precise definition and assumptions are given in \Cref{sec:results}.  Here, we only remark that, if $w$ is mild white noise, then, as $\e \to 0$, $\e^{-1}\int_0^y w(z/\e^2) dz$ converges in distribution to a Brownian motion.


\medskip

The random advection whose effect we investigate is
\[
	u(x,y,t) = (u_\perp(x,y,t), u_\parallel(x,t) w(y)),
\]
where $u_\perp$ and $u_\parallel$ are smooth and bounded.  We study fronts that, on average, propagate in the $y$-direction, so that $u_\perp$ and $u_\parallel w$ are the perpendicular and parallel advective forces respectively.

\medskip

To state the results, we define two objects that will be of considerable importance to our study since they provide the correction due to the small advection.  For a fixed standard one-dimensional Brownian motion $W$, we consider the stochastic Hamilton-Jacobi equation
\begin{equation}\label{eq:corrector}
 	\begin{cases}
 		d \chi + \frac{1}{2} |\nabla_x \chi|^2 d\xi = - u_\parallel(\xi,0) dW(\xi) \qquad &\text{ in }~~ \R^{n-1} \times \R_+,\\
 		\chi = 0 &\text{ on }~~ \R^{n-1}\times\{0\}.
	\end{cases}
\end{equation}
and its viscous counterpart
\begin{equation}\label{eq:corrector_viscous}
 	\begin{cases}
 		d\chi_{\rm visc} + \left(\frac{1}{2} |\nabla_x \chi_{\rm visc}|^2  - \frac{1}{2} \Delta_x \chi_{\rm visc}\right) d\xi = - u_\parallel(\xi,0) d W(\xi) \qquad &\text{ in }~~ \R^{n-1}\times\R_+,\\
 		\chi_{\rm visc} = 0 &\text{ on }~~ \R^{n-1}\times\{0\}.
	\end{cases}
\end{equation}
%
%
%
Because of the lack of regularity of $dW$ in~\eqref{eq:corrector} and~\eqref{eq:corrector_viscous}, the classic notion of viscosity solution is not applicable here.  At the end of \Cref{sec:results}, we explain how to make sense of~\eqref{eq:corrector} and~\eqref{eq:corrector_viscous}.

\medskip

Next, we introduce the models and describe the results.

\subsubsection*{The G-equation}

We fix $\alpha \geq 1$ and consider the initial value problem 
\begin{equation}\label{eq:geometric}
	\begin{cases}
		G^\epsilon_t + \e u(x,y,\e^\alpha t) \cdot \nabla G^\epsilon + |\nabla G^\epsilon| = 0 \quad &\text{ in }~~ \R^n\times \R_+,\\
		G^\epsilon = G_0 & \text{ on }~~ \R^n \times \{0\},
	\end{cases}
\end{equation}
where $G_0$ is a ``front-like'' initial datum (see \Cref{assumption:initial_datum}), the simplest example being $G_0(x,y) = y$.  We are interested in the evolution of the``front,'' that is, the $0$-level set of $G^\e$ at time $t$, which we denote $\Gamma_t(G^\e)$.  We note that, if $\e=0$ and $G_0(y) = y$, then $G^0(x,y,t) = y-t$ solves~\eqref{eq:geometric}, and its front at time $t$ is given by $\Gamma_t(G^0) = \{(x,y): y = t\}$.  Our goal is to understand in what way it is approximated by the front of $G^\e$.

\medskip

The case $\alpha = \infty$ is allowed, and the convention is that $\e^\infty = 0$.

\medskip

The first result is stated informally in the following theorem.  The precise statements are given  in \Cref{thm:geometric_ti} and \Cref{prop:geometric_td}.
\begin{theorem}\label{thm:rough_geometric}
	If $G^\e$ solves~\eqref{eq:geometric} and $G_0$ is front-like, then
	\[
		\Gamma_t(G^\e)
			= \{(x,y)\in \R^n : y + \e^{2/3} \chi^\e \left( x, \e^{2/3} y, \e^{2/3} t\right) = t\},
	\]
	 where, as $\e\to0$, $\chi^\e$ converges in distribution to the solution $\chi$ of~\eqref{eq:corrector}.
\end{theorem}

\subsubsection*{The eikonal equation }

The second model is
\begin{equation}\label{eq:FKPP}
\begin{cases}
	v_t^\e + \e u(x,y, \e^\alpha t) \cdot \nabla v^\e + \frac{1}{2} |\nabla v_x^\e|^2 + \frac{1}{2} = \frac{\e^\beta}{2} \Delta v^\e \qquad &\text{ in }~~ \R^n \times \R_+,\\
	v^\e = v_0 & \text{ on }~~ \R^n \times \{0\},
\end{cases}
\end{equation}
where $v_0$ is front-like.  For the sake of completeness, we describe the connection of~\eqref{eq:FKPP} to a turbulent reaction-diffusion equation.  A simple calculation yields that $T^\e(x,y,t) := \exp\{ - \e^{-\beta} v^\e(\e^{\beta} x, \e^\beta y, \e^\alpha t)\}$ solves
\begin{equation}\label{eq:FKPP_linearized}
	T_t^\e + u \cdot \nabla T^\e
		= \frac{1}{2} \Delta T^\e + \frac12 T^\e.
\end{equation}
The front of $T^\e$ is the area where it transitions from $T^\e \approx 0$ to $T^\e \approx O(1)$.  It is clear from the relationship between $T$ and $v$ that the two uses of the term ``front'' are consistent.  When $u \equiv 0$, the front of $T$ is approximately the same as those of solutions of the Fisher-KPP equation, which is sometimes used as a model for combustion.

\medskip

Our second result is stated informally in the following theorem.  The precise statement can be found in \Cref{thm:FKPP} and \Cref{prop:ivp}.
\begin{theorem}\label{thm:rough_FKPP}
If $v^\e$ solves~\eqref{eq:FKPP} and $v_0$ is front-like, then
	\[
		\Gamma_t(v^\e) \approx
			\{(x,y) \in \R^n: y + \e^{2/3} \chi^\e(x, \e^{2/3} y, \e^{2/3} t) = t\},
	\]
	where, as $\e\to0$, $\chi^\e$ converges in distribution to the solution $\chi$ of~\eqref{eq:corrector} when $\beta > 2/3$ and to the solution $\chi_{\rm visc}$ of~\eqref{eq:corrector_viscous} when $\beta = 2/3$. 
\end{theorem}

We point out that the front location for the G-equation, given in \Cref{thm:rough_geometric}, and those of the eikonal equation, given in \Cref{thm:rough_FKPP}, have the same approximate expansion,
\[
	y + \e^{2/3} \chi(x, \e^{2/3} y) + (\text{lower order terms})
		= t.
\]
This is somewhat surprising since examples were given in~\cite{MajdaSouganidis} where these two models do not have the same front asymptotics for $\e>0$.

\subsubsection*{A simple example}

To illustrate the results, we find the front in the simple example where $u_\parallel \equiv 1$.  Since the conclusion is the same for both $G^\e$ and $v^\e$, we consider, for notational simplicity, only the solution $G^\e$ of~\eqref{eq:geometric};  however, the same discussion applies to the solution $v^\e$ of~\eqref{eq:FKPP}.  With $u_\parallel \equiv 1$, the solution to~\eqref{eq:corrector} is $\chi(x,\xi) = W(\xi)$.  \Cref{thm:rough_geometric} yields that the front location is
\[
	\Gamma_t(G^\e)
		= \{(x,y) \in \R^n: t = y + \e^{2/3} \chi^\e(x,\e^{2/3}y)\}
		\approx \{(x,y) \in \R^n: t = y + \e^{2/3} W(\e^{2/3}y)\}.
\]
Since, in view of the Brownian scaling, $\e^{2/3} W(\e^{2/3}y)$ is equal in distribution to $\e \sqrt y (\widetilde W(t) /\sqrt t)$, where $\widetilde W$ is an independent Brownian motion, we find that $(x,y)$ belongs to the front at time $t$ when $t \approx y + \e \sqrt y \widetilde W(t)/\sqrt t$, that is
\[
	\Gamma_t(G^\e)
		\approx \{(x,y) \in \R^n: y
		\approx t - \e \widetilde W(t)\}.
\]
In other words, we see Brownian fluctuations of the front of order $\e$.

\subsubsection*{Further connections with previous works}

In addition to the related work discussed above, our work is placed in the field of research into precise descriptions of the effect of advection on front propagation.  The body of literature devoted to these problems is enormous, and we thus only provide a small sample of the current research that is most relevant to the current work.  While certain implicit representation formulas of the speed and the front profile exist (see, e.g., Xin~\cite{Xin09}), they are often difficult to quantify.  To our knowledge, most non-trivial results that can be quantified precisely are done in particular asymptotic regimes, especially when the flow becomes large.  We mention the studies of reaction-diffusion equations in the presence of a large time-independent shear flow by Hamel and Zlatos~\cite{HamelZlatos} and a large cellular flow by Novikov and Ryzhik~\cite{NovikovRyzhik}.  In addition, Hamilton-Jacobi models like~\eqref{eq:FKPP} and~\eqref{eq:geometric} have been studied in the setting with a large cellular flow by Xin and Yu~\cite{XinYu} and when $u$ is the ABC flow by Xin, Yu, and Zlatos~\cite{XinYuZlatos}.

%

\medskip

Beyond this, we mention a somewhat surprising connection to a recent work by Corwin and Tsai on the weakly inhomogeneous ASEP process~\cite{CorwinTsai}.  There, using probabilistic techniques, the authors show that the introduction of a small inhomogeneity yields fluctuations around the homogeneous process that are governed by an equation similar to~\eqref{eq:corrector_viscous} (see \cite[equation (1.7) and Remark 1.8]{CorwinTsai}).  To roughly see why the two results should be related, one should think of the inhomogeneity in their process as a random drift term, similar to $u$.

\subsubsection*{Organization of the paper}

The assumptions and results are stated  more precisely in \Cref{sec:results}.  
In \Cref{sec:special} we construct some special solutions that we refer to as ``perturbed traveling waves.''  We do this first in the autonomous setting and then extend it by a bootstrapping argument to the non-autonmous problem.  These results are then used in \Cref{sec:G_ivp} and \Cref{sec:FKPP_ivp} to understand the front location for the initial value problems~\eqref{eq:geometric} and~\eqref{eq:FKPP} respectively.  This allows us to conclude the proofs of \Cref{thm:rough_geometric} and \Cref{thm:rough_FKPP}. The main technical lemma that we use to construct the perturbed traveling waves is the a priori estimates on the metric planar problem.  This is the subject of \Cref{sec:rho}.  

\subsubsection*{Acknowledgements}

Henderson was partially supported by the National Science Foundation Research Training Group grant DMS-1246999. Souganidis was partially supported by the National Science Foundation grants DMS-1266383 and DMS-1600129 and the Office for Naval Research Grant N00014-17-1-2095.

\section{Assumptions and Results}\label{sec:results}

\subsection{The assumptions}
We begin with the assumptions on the initial datum and the advection.  The first, which concerns~\eqref{eq:geometric} and~\eqref{eq:FKPP}, is that, heuristically, the $0$-level set of $G_0$ is $\{y = 0\}$ and $G_0$ ``lifts'' away from zero in a uniform way in $x$ (see \Cref{fig:G_0_assumption}).  The latter is assumed to avoid ``fattening'' of the $0$-level set as $|x|\to\infty$.  
For a more in-depth discussion of the level set method and issues related to fattening, we refer the reader to the review by Souganidis~\cite{SouganidisCIMA}.
\begin{assumption}\label{assumption:initial_datum}
$G_0 \in L^\infty_{\rm loc}(\R^n)$ and there exist $\underline G,\, \overline G \in C^{0,1}_{\rm loc}(\R)\cap C^1_{\rm loc}(\R_-\cup\R_+)$ such that $\underline G', \overline G' > 0$, $\underline G \leq G_0 \leq \overline G$, and $\overline G(0) = \overline G (0) = 0$.
\end{assumption}
Initial data satisfying~\Cref{assumption:initial_datum} are sometimes called ``front-like.''  The prototypical example is $G_0(x,y) = y$.
\begin{figure}
\begin{center}
\begin{overpic}[scale = .5]
		{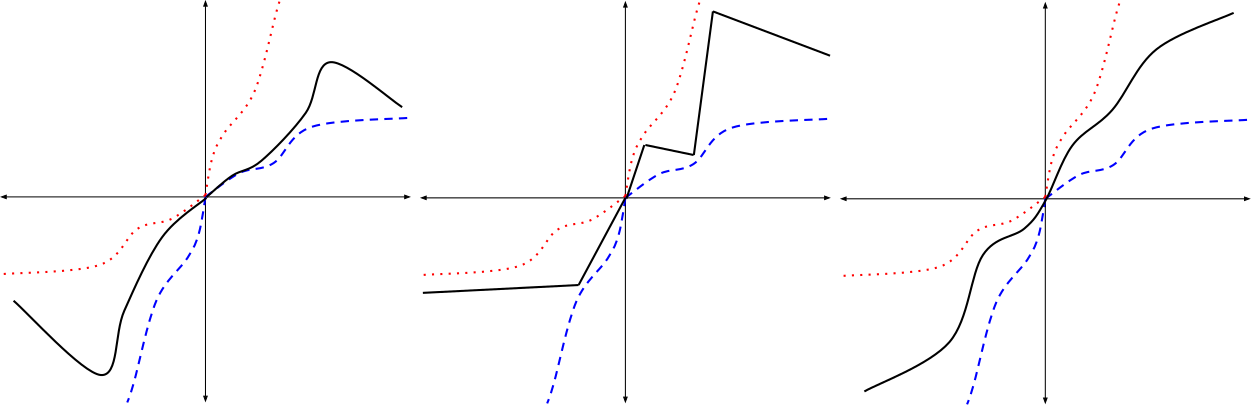}
	\put (98.5,14.5){$y$}	
	\put (28,28){$G_0(x_1,\cdot)$}	
	\put (62,30.5){$G_0(x_2,\cdot)$}
	\put (95, 28){$G_0(x_3,\cdot)$}
	\put (2,12) {\color{red} $\overline G$}
	\put (12,3) {\color{blue} $\underline G$}
\end{overpic}
\caption{A cartoon illustrating \Cref{assumption:initial_datum}.  Each plot is the profile of $G_0(x_i,\cdot)$ for three values  $x_1$, $x_2$, $x_3 \in \R^{n-1}$.  The dotted line is $\overline G$, the  dashed line is $\underline G$, and the solid black line is $G_0(x_i,\cdot)$.  Notice that, regardless of $x_i$, $G_0$ leaves zero at $y= 0$ in a uniform way.}\label{fig:G_0_assumption}
\end{center}
\end{figure}
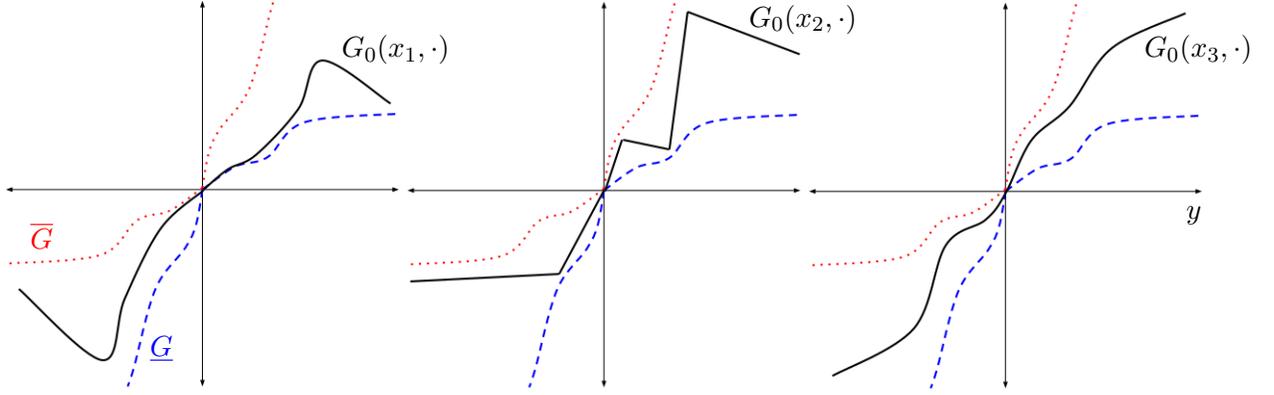

\medskip

Before we state the assumption on the advection, we discuss the notion of mild approximation of white noise.  Let $(\Omega, \mathcal{F}, P)$ be a probability space with expectation $E$, and let $\mathcal{F}_{y_1,y_2}:= \sigma\{w(y): y_1 \leq y \leq y_2\}$.  We say that $w: \R\times \Omega \to \R$ is a mild approximation of white noise if
\begin{enumerate}[(i)]
\item there exists $M>0$ such that, with probability $1$, $\|w\|_{C^1(\R)} \leq M$;
\item for all $y\in \R$, $\E[w(y)] = 0$;
\item $w$ is stationary and strongly mixing with rate $p>3/2$; that is, if
\[
	\rho(y) := \sup_{y_1} \sup_{y_2 \geq y_1}  \sup_{A \in \cF_{y_2 + y, \infty}, B\in \cF_{y_1, y_2}} \frac{|P(A\cap B) - P(A)P(B)|}{P(B)},
	\]
	then
\[
	\int_0^\infty \rho(y)^{1/p} dy < \infty.
\]
\end{enumerate}
To simplify the notation, in what follows, we assume that $M\geq 1$ and 
\[
	2 \int_0^\infty E[w(0) w(\xi)] d\xi =1.\]

\medskip

It is well-known that, if $w$ satisfies (i), (ii), and (iii), then
\begin{equation}\label{eq:approximate_BM}
	W^\e(y) := \e^{-1/3} \int_0^y w( \e^{-2/3} z) dz
\end{equation}
converges, as $\e\to0$, in distribution to a Brownian motion $W$; see, for example,  Funaki~\cite{Funaki}.  The term mild refers to the lower bound on $p$ in (iii).   
For an more extensive discussion about mild approximation of white noise, we refer to Ikeda and Watanabe~\cite{IkedaWatanabe}.

\medskip

A simple example of mild white noise $w$ is
\begin{equation}\label{eq:advection_example}
	w(y) = \int_\R \widetilde S_{z} \phi'(y-z) dz,
\end{equation}
where $\widetilde S$ is a piece-wise linear interpolation of a random walk $S$, indexed by $\Z$ and with $S_0 = 0$, and $\phi \in C^\infty_c$ is non-negative and $\supp(\phi) \subset [0,1]$.  
Properties (i) and (ii) are clearly satisfied, while (iii) is verified by writing
\[
	w(y) = \int_{y-1}^y \left(\widetilde S_z - \widetilde S_{y-1}\right) \phi'(y - z) dz,
\]
noticing that $w(y)$ and $w(y')$ are independent if $y'>y+1$, and observing that $\supp \rho \subset [0,1]$.

\medskip

The second assumption is:
\begin{assumption}\label{assumption:advection}
The advection $u$ is of the form
\begin{equation}
	u(x, y, t)
		= (u_\perp(x,y, t), u_\parallel(x, t) w(y)),
\end{equation}
where $w$ is mild white noise, $u_\perp  \in C^2(\R^n \times \R_+)^{n-1}$, and $u_\parallel \in C^2(\R^{n-1}\times \R_+)$.  
\end{assumption}
We are interested in the fronts $\Gamma_t(G^\e)$ and $\Gamma_t(v^\e)$ of $G^\e$ and $v^\e$ respectively, where, for any $\phi : \R^n\times\R_+ \to \R$ and $t\in \R_+$, 
\begin{equation}\label{eq:front}
	\Gamma_t(\phi) := \{(x,y) \in \R^n : \phi(x,y,t) = 0\}.
\end{equation}
As discussed above, a special solution of~\eqref{eq:geometric} and~\eqref{eq:FKPP}, when $\e=0$, is $G^0(x,y,t) = v^0(x,y,t) = y - t$. Hence, $\Gamma_t(G^0) = \Gamma_t(v^0) = \{(x,t): x \in \R^{n-1}\}$.  The goal is to understand the first order correction to this for $\e \ll 1$.

\subsection{The G-equation}

We first construct a special solution of~\eqref{eq:geometric} that has the form $y-t + \e^{2/3} \chi^\e$ and that we refer to as a ``perturbed traveling wave''.  We use this term for two reasons.  Firstly, it is the sum of a traveling wave $y-t$ and a small term $\e^{2/3} \chi^\e$, and secondly, it is a special solution that plays a fundamental role in analyzing the general case, much like a traveling wave.   The perturbation $\chi^\e$ acts as the ``corrector'' in the averaging problem that we are studying.

\begin{theorem}\label{thm:geometric_ti}
	Suppose that \Cref{assumption:advection} holds and $\alpha \geq 1$.  There exists $\chi^\e \in L^\infty_{loc}(\R^n\times\R_+)$ such that
	\begin{enumerate}[(i)]
	\item $G_{\ptw}^\e(x,y,t) := y - t + \e^{2/3} \chi^\e(x,\e^{2/3} y, \e^{2/3} t)$ solves~\eqref{eq:geometric},
 	\item  $\chi^\e$ converges in distribution on $\R^{n-1}\times\{(\xi,\tau) \in \R\times [0,\infty): \xi \geq \tau\}$, as $\e \to 0$, to the solution $\chi$ of~\eqref{eq:corrector},
 	\item $G^\e_\ptw(\cdot,\cdot,0)$ satisfies \Cref{assumption:initial_datum}.
	\end{enumerate}
\end{theorem}

Clearly $G^\e_\ptw$ depends on $\alpha$, but we omit this for notational simplicity.

\medskip

We describe and discuss the precise definition of locally uniform convergence on $\R^{n-1}\times\{(\xi,\tau) \in \R \times [0,\infty) : \xi \geq \tau\}$ that we use throughout at the end of this section.

\medskip

Although the convergence of $\chi^\e$ to $\chi$ holds on $\R^{n-1}\times \{(\xi,\tau) \in \R\times [0,\infty): \xi \geq \tau\}$,  the relevant set for locating the front is merely $\R^{n-1}\times \{(\xi,\xi) : \xi \in [0,\infty)\}$.  To see this, notice that
\[
	\Gamma_t(G^\e_\ptw) = \{(x,y,t) : y + \e^{2/3} \chi^\e(x, \e^{2/3}y, \e^{2/3} t) = t\}.
\]
It follows from the a priori estimates~\eqref{eq:chi_bound} on $\chi^\e$ that $y = t + o(1)$, where $o(1) \to 0$ as $\e\to 0$.  Letting $\xi = \e^{2/3} x$ and $\tau = \e^{2/3} t$, the term involving the corrector becomes $\e^{2/3} \chi^\e(x, \tau + o(1), \tau)$.  It is thus apparent that, to understand the front location when $\e \ll 1$, it is sufficient to study the convergence of $\chi^\e(x,\xi,\tau)$ when $\xi = \tau + o(1)$.

\medskip

We note the interesting fact that the transverse advection $u_\perp$ does not affect the first order correction in the limit.  In addition, we point out that while $\chi^\e$ has time-dependence for all $\e>0$, it converges to a limit $\chi$ that does not evolve in time.  Finally, we remark that we do not know if the restriction $\alpha \geq 1$ is sharp.

\medskip

One way to understand \Cref{thm:geometric_ti} is through the following informal computation that ignores technical issues such as the time dependence of $u$ and the lack of regularity of $G^\e_\ptw$.  When $\alpha = \infty$, we use the ansatz
\[
	G^\e_\ptw(x,y,t) = y - t + \e^{2/3} \chi^\e(x,\e^{2/3} y),
\]
which, from~\eqref{eq:geometric}, yields
\[
	1 = \e u(x, \e^{-2/3} \xi) \cdot (\e^{2/3} D_x \chi^\e, 1 + \e^{4/3} \chi^\e_\xi) + |(D_x \chi^\e, 1 + \e^{4/3} \chi^\e_\xi)|.
\]
Approximating the last term with a Taylor expansion yields
\[
	1 = \e u_\parallel(x) w(\e^{-2/3} \xi)  + \frac{1}{2}|D_x \chi^\e|^2 + 1 + \e^{4/3} \chi^\e_\xi + O(\e^{5/3}).
\]
Re-arranging, dividing by $\e^{4/3}$, and using that $\e^{-1/3} w(\e^{-2/3} \xi) = W_\xi^\e(\xi)$, we find
\[
	\chi^\e_\xi + \frac12 |D_x \chi^\e|^2 = - u_\parallel W_\xi^\e + O(\e^{1/3}).
\]
We identify~\eqref{eq:corrector} by taking the limit $\e\to0$.

\medskip

Using the level set method, we can describe the front asymptotics for solutions $G^\e$ of~\eqref{eq:geometric} with more general initial datum.
\begin{proposition}\label{prop:geometric_td}
	Suppose that $\alpha \geq 1$, and let $G^\e$ solve~\eqref{eq:geometric} with $G_0$ and $u$ satisfying \Cref{assumption:initial_datum} and \Cref{assumption:advection} respectively.  Then, for all $t\in \R_+$, $\Gamma_t(G^\e) = \Gamma_t(G^\e_\ptw).$
	Moreover, $\{G^\e \leq 0 \} = \{G^\e_{\ptw} \leq 0\}$.
\end{proposition}

%

\Cref{prop:geometric_td} implies that the special solutions constructed in \Cref{thm:geometric_ti} are sufficiently stable to determine the front for the general initial value problem.  

%

%
%
%
%
%

\subsection{The eikonal equation}

As above, we begin by constructing the perturbed traveling waves for~\eqref{eq:FKPP}, that is, we state the analogue of \Cref{thm:geometric_ti}.
\begin{theorem}\label{thm:FKPP}
Suppose that \Cref{assumption:advection} holds, $\alpha \geq 1$, and $\beta \geq 2/3$.  There exists $\chi^\e \in L^\infty_{loc}(\R^n\times\R_+)$ such that
	\begin{enumerate}[(i)]
	\item $v_{\ptw}^\e(x,y,t) := y - t + \e^{2/3} \chi^\e(x,\e^{2/3} y, \e^{2/3} t)$ solves~\eqref{eq:FKPP},
 	\item $\chi^\e$ converges in distribution on $\R^{n-1}\times \{(\xi,\tau) \in \R\times[0,\infty): \xi \geq \tau\}$, as $\e \to 0$, to the solution $\chi$ of~\eqref{eq:corrector} when $\beta > 2/3$ and the solution $\chi_{\rm visc}$ of~\eqref{eq:corrector_viscous} when $\beta = 2/3$,
 	\item $v^\e_\ptw(\cdot,\cdot,0)$ satisfies \Cref{assumption:initial_datum}.
	\end{enumerate}
\end{theorem}

We note that $\beta =2/3$ is the critical scale in order to see the effect of the viscosity in the limit.  

\medskip

It is harder to bootstrap the front asymptotics of the perturbed traveling wave since the level set method only works for positively homogeneous equations of degree one.  Hence, we obtain estimates on the 0-sub-level set, which, while quite sharp, do not completely characterize the $0$-level set as in~\Cref{prop:geometric_td}.

\begin{proposition}\label{prop:ivp}
Assume that $\beta = \infty$ and $\alpha \geq 1$.  Suppose that $v_0$ and $u$ satisfy \Cref{assumption:initial_datum} and \Cref{assumption:advection} respectively,  $v_0 \geq v^\e_\ptw(\cdot,\cdot,0)$ in $\R^n$, and  $v^\e$ solves~\eqref{eq:FKPP}.  Then
	\begin{equation}\label{eq:sub-level}
		\{(x,y) : G^\e_{\rm ptw} (x,y,t) \leq 0 \}
			\subset \{(x,y) : v^\e(x,y,t) \leq 0\} 
			\subset \{(x,y) : v^\e_{\rm ptw}(x,y,t) \leq 0\}.
	\end{equation}
\end{proposition}

In view of \Cref{thm:FKPP} and \Cref{thm:geometric_ti}, this result indicates that $v^\e$ has the same front expansion in terms of $\chi$ at the $\e^{2/3}$-order.

\medskip

The extra condition on the initial datum in \Cref{prop:ivp} is quite sharp.  Indeed, fix any $\mu > 1$ and consider the solution of~\eqref{eq:FKPP} with initial datum $v_0(x,y) = y/\mu$.  Letting $\underline v(x,y,t) = -t(\kappa + \e\|u\|_\infty) + y/\mu$ and $\overline v(x,y,t) = -t(\kappa - \e\|u\|_\infty) + y/\mu$, where $\kappa = (2\mu^2)^{-1} + (1/2)$, we see that $\underline v$ and $\overline v$ are, respectively, sub- and super-solutions of~\eqref{eq:FKPP}.  Applying then the comparison principle, we find $\underline v \leq v^\e \leq \overline v$, and, hence, we conclude that
\[
	(x,y) \in \Gamma_t(v^\e)
		\Leftrightarrow
		y \approx \mu \kappa t + O(\e t).
\]
After noting that $\mu\kappa > 1$, this indicates that the sub-level sets of $v^\e$ with this initial datum cannot satisfy~\eqref{eq:sub-level}.

\subsection{Discussion of the proofs, organization, and notation}

\subsubsection*{Discussion of the proof and main difficulties}



The first step is to construct the perturbed traveling waves in the autonomous setting ($\alpha = \infty$). As discussed heuristically below \Cref{thm:geometric_ti}, the proof proceeds via an ansatz that $G^\e_\ptw$ and $v^\e_\ptw$ are of the form $-t + \rho^\e$, where $\rho^\e$ is time-independent and solves the so-called metric planar problem.  We expect the expansion $\rho^\e(x,y) = y + \e^{2/3} \chi^\e(x, \e^{2/3} y)$.  Defining $\chi^\e$ in this way, we use the half-relaxed limits in order to take limit as $\e\to 0$.  Informally, the half-relaxed limits are the ``smallest supersolution'' below $\rho^\e$ and the ``largest subsolution'' above $\rho^\e$ as $\e\to0$.  It can often be shown, using the comparison principle, that these two objects coincide.

\medskip

The latter requires to overcome two main difficulties.  The first is that the process $W^\e$ converges, as $\e\to0$, to $W$ only in distribution.  This does not interact well with the half-relaxed limits, which require pointwise convergence.  To get around this obstruction, we use an argument from~\cite{IkedaWatanabe} that allows to replace $W^\e$ with a process $\widetilde W^\e$ that converges, as $\e\to0$, almost surely to a standard Brownian motion and equals $W^\e$ in distribution.  The second major difficulty is how to obtain a priori estimates of $\rho^\e$ that are sufficiently sharp to conclude that $\rho^\e = y + \e^{2/3} \chi^\e$, where $\chi^\e$ is bounded and $\lim_{\e\to0} \chi^\e_\ptw$ satisfies the correct datum at $y=0$.  This is achieved through the construction of suitable barriers.

\medskip

The above strategy is not enough to study the non-autonomous problem, that is, when $\alpha < \infty$, due to the time-dependence inherited in the equation for $\rho^\e$.  Roughly speaking, our strategy is to build  the perturbed traveling wave in this setting by the addition of a ``very small'' correction term to the perturbed traveling wave from the autonomous case.

\medskip

More specifically, we define the perturbed traveling waves for the non-autonomous problem to be the solutions of~\eqref{eq:geometric} and~\eqref{eq:FKPP} with initial datum that is equal to the perturbed traveling wave from the autonomous case.  We are then able to obtain sufficiently good error estimates between the solution and its initial data allowing to take the half-relaxed limits as $\e\to0$. The result is a non-standard, non-coercive Hamilton-Jacobi equation solved by both the limit and $\chi$ for $\xi > 0$.

\medskip

We do not, however, have control on $\chi$ and the half-relaxed limits $\chi^*$, $\chi_*$  for $\xi < 0$.  The standard comparison principle is valid for for this equation but requires information about $\chi$, $\chi^*$, and $\chi_*$ on $\{\xi < 0\}$.  We side-step this by using a simple change of variables that allows to compare solutions on sets that are preserved by the characteristics, that is, where $\xi -\tau$ is constant. We are thus able to conclude the convergence to $\chi$ in this setting.

\medskip

We bootstrap the results above to general initial datum.  We can conclude \Cref{prop:geometric_td} using the level set method.  In addition, we prove \Cref{prop:ivp} by using the perturbed traveling waves of \Cref{thm:geometric_ti,thm:FKPP} to construct sub- and super-solutions of $v^\e$.

\subsubsection*{Additional notation}

Throughout we only work with locally uniform convergence on sets of the form $\R^{n-1}\times \{(\xi,\tau) \in \R\times [0,\infty): \xi \geq \tau\}$.  Since we care about endpoint behavior at $\xi = \tau$, we use a slightly stronger notion than the standard one.  Indeed, we say that $f_n$ converges to $f$ locally uniformly on $\R^{n-1}\times \{(\xi,\tau) \in \R\times [0,\infty): \xi \geq \tau\}$ if, for any $(x_0,\xi_0,\tau_0) \in \R^{n-1}\times \{(\xi,\tau) \in \R\times [0,\infty): \xi \geq \tau\}$ and any sequence $(x_n, \xi_n, \tau_n) \in \R^n \times [0,\infty)$ converging, as $n\to\infty$, to $(x_0,\xi_0,\tau_0)$, we have $f_n(x_n,\xi_n,\tau_n) \to f(x_0,\xi_0,\tau_0)$ as $n\to\infty$.  The difference is that we allow each $\xi_n$ to take any real values, instead of just values in $[\tau_n,\infty)$.

\medskip

For any $f\in C^{0,1}(\R^n)$, $\Lip(f)$ denotes its Lipschitz constant, for any $f\in L^\infty(\R^n)$, $\|f\|_\infty$ denotes its $L^\infty$-norm, and, for any $f\in C^1(\R^n)$, $\|f\|_{C^1}$ denotes its $C^1$-norm.  Also, $\delta$ denotes the Kronecker delta function.

\medskip

Since we are concerned with the small $\e$ limit, we lose no generality in assuming throughout the paper that $\epsilon \|u\|_{C^1} \leq 1/100$.


\medskip

All functions throughout depend on the variable $\omega \in \Omega$.  When no confusion arises, we suppress this dependence to simplify the writing.

\medskip

Given random variables $X_1, X_2,\dots$ and $X$, $X_n \tod X$ and $X_n \toas X$ mean that, as $n\to\infty$, $X_n$ converges to $X$ in distribution and almost surely respectively.  When two random variables $X$ and $\widetilde X$ have the same distribution, we write $X \stackrel{d}{=} \widetilde X$.

\medskip

Throughout the paper, $W$ is a one-dimensional standard Brownian motion and $W(\xi)$ denotes the value of $W$ at $\xi$.  In addition, we denote white noise by $dW$.   It is important to note that this is one-dimensional white noise in the variable $\xi$ and not space-time white noise.

\medskip

We now make explicit the notion of solution of equations of the form
\begin{equation}\label{eq:svs}
	d f + (H(\nabla f) - \nu \Delta f) dt = g dW(t),
\end{equation}
where $H$ is some Hamiltonian and $\nu \geq 0$. 
We say that $f$ is a solution of~\eqref{eq:svs} if and only if $\overline f(x,t) = f(x,t) - g(x) W(t)$ is a viscosity solution of
\begin{equation}\label{eq:svs_def}
	\overline f_t + H\left( \nabla \overline f + W(t) \nabla g\right)  - \nu \Delta\left(\overline f + W(t) g\right)= 0.
\end{equation}
This definition was used by Dirr and Souganidis in~\cite{DirrSouganidis} and is a special case of the general notion of solution introduced by Lions and Souganidis in~\cite{LionsSouganidis1, LionsSouganidis2, LionsSouganidis3}.

\section{The construction of the perturbed traveling waves}\label{sec:special} 

We prove \Cref{thm:geometric_ti,thm:FKPP}.  Since the arguments are similar, we reduce them to a more general claim (see \Cref{prop:general}).  We begin by addressing the autonomous case $\alpha = \infty$.  Then, we bootstrap to the non-autonomous case (see \Cref{prop:general_td}).

\subsection{The autonomous case $\alpha = \infty$}\label{sec:autonomous}

We work in a more general framework and state the main claim next.
\begin{prop}\label{prop:general}
Suppose that \Cref{assumption:advection} holds, $\beta \geq 2/3$, and $r\in [1,2]$.  There exists $\chi_\aut^\e \in L^\infty_{loc}(\R^n)$ such that
	\begin{enumerate}[(i)]
	\item $f_\aut^\e(x,y,t) := y - t + \e^{2/3} \chi_\aut^\e(x,\e^{2/3} y)$ solves	\begin{equation}\label{eq:general_problem}
		f_{\aut,t}^\e + \e u_\aut \cdot \nabla f_\aut^\e + \frac{1}{r} |\nabla f_\aut^\e|^r + \frac{r-1}{r} = \frac{\e^\beta}{2} \Delta_x f_\aut^\e \qquad \text{ in }~~ \R^n\times \R_+,\\
	\end{equation}
	where $u_\aut^\e(x,y) := u(x,y,0)$;
 	\item as $\e \to 0$, $\chi_\aut^\e$ converges in distribution  on $\R^{n-1}\times [0,\infty)$ to $\chi$, the unique solution of~\eqref{eq:corrector}, if $\beta > 2/3$, or $\chi_{\rm visc}$, the unique solution of~\eqref{eq:corrector_viscous}, if $\beta = 2/3$;
 	\item $f^\e_\aut(\cdot,\cdot,0)$ satisfies \Cref{assumption:initial_datum}.
	\end{enumerate}
%
%
%
\end{prop}
The reason for the restriction $r \leq 2$ is seen in the a priori estimates of $\chi^\e$.  While we do not anticipate any issues in extending the proof to $r>2$,  this will involve some adjustments to our proof.  Since our interest is in the cases $r=1,2$, we opt for a simpler presentation and, thus, restrict to $r\in [1,2]$.

\medskip

The proof proceeds in several steps.  First we reduce to an intermediate model using the ansatz that $f^\e = - t + \rho^\e$ for a time-independent $\rho^\e$ solving the so-called metric planar problem.  Then, we extract $\chi^\e$ from $\rho^\e$ and reduce to the stronger case where $W^\e$ converges in probability to $W$.  Finally, we apply the method of half-relaxed limits to obtain convergence of $\chi^\e$ to $\chi$.

\subsubsection{Step (i): the reduction to a time-independent problem}

From the form of the claim, it is natural to seek a solution $f_\aut^\e(x,y,t) := \rho^\e(x,y) - t$, where $\rho^\e$ solves
\begin{equation}\label{eq:general_metric_problem}
\begin{cases}
	- r \frac{\e^\beta}{2} \Delta \rho^\e + r \e u_\aut \cdot \nabla \rho^\e + |\nabla \rho^\e|^r = 1 \qquad &\text{ in }~~ \R^n,\\
	\rho^\e = 0 &\text{ on }~~ \R^{n-1}\times\{0\}.
\end{cases}
\end{equation}
%

Next, we consider the existence, uniqueness, and some a priori bounds of $\rho^\e$.
\begin{lemma}\label{lem:rho}
	There exists a unique globally 
	Lipschitz solution $\rho^\e$ to~\eqref{eq:general_metric_problem} such that, uniformly for all $x \in \R^{n-1}$,
	\[\begin{split}
		\liminf_{y\to\infty} \rho^\e(x,y) \geq 0,
			\qquad \text{and} \qquad
		\limsup_{y\to-\infty} \rho^\e(x,y) \leq 0.
	\end{split}\]
	Moreover, for all $(x,y) \in \R^n$,
	\begin{equation}\label{eq:bounds_rho1}
		|\rho(x,y) - y| \leq |y|/2,
	\end{equation}
	and there exist $C_L$, $\mu_1$, $\mu_2$, and $\mu_3$, depending only on $\|u\|_{C^1}$ and $M$, such that $\Lip(\rho^\e)\leq C_L$, and, for all $(x,y) \in \R^n$,
	\begin{equation}\label{eq:bounds_rho}
		\left|\rho^\e(x,y) - \left(y - \e^{2/3} u_{\parallel} W^\e(\e^{2/3} y)\right)\right|
			\leq \e^{4/3} \mu_1 |y| + \frac{\mu_2 \e^2 y^2}{2}
				+ \e^{2/3} \mu_3 \bigg|\int_0^{y \e^{2/3}} |W^\e(y')|^2 dy'\bigg|.
	\end{equation}
\end{lemma}

The existence and uniqueness of $\rho^\e$ is well-understood because problems like~\eqref{eq:general_metric_problem} have been studied extensively due to their use in stochastic homogenization; see, for example, the work of Armstrong and Cardaliaguet~\cite{ArmstrongCardaliaguet}, Armstrong, Cardaliaguet, and Souganidis~\cite{ArmstrongCardaliaguetSouganidis}, and Armstrong and Souganidis~\cite{ArmstrongSouganidis}, and references therein.  The sharp bound~\eqref{eq:bounds_rho} in \Cref{lem:rho}, which justifies the earlier comment about correctors, is new and requires significant effort.  The construction of sufficiently sharp sub- and super-solutions is quite involved.  The proof of \Cref{lem:rho} is presented in~\Cref{sec:rho}.

\medskip

The motivation for the weaker bound~\eqref{eq:bounds_rho1} is two-fold.  Firstly, it shows that $f^\e_\aut(\cdot,\cdot,0)$ satisfies \Cref{assumption:initial_datum}.  Secondly, it is used in the proof of \Cref{prop:geometric_td}.  Note that~\eqref{eq:bounds_rho1} does not follow from the sharper bound~\eqref{eq:bounds_rho} due to the behavior for $|y|\gg 1$.    The sharper bound is a crucial part of the proof of \Cref{prop:general}.

\subsubsection{Step (ii): the extraction of the correctors $\chi^\e_\aut$}

We change variables so that $\xi = y\e^{2/3}$ and let, for all $(x,\xi) \in \R^n$,
\begin{equation}\label{eq:chi_definition}
	\chi_\aut^\e(x,\xi)
		:= \frac{\rho^\e(x, \e^{-2/3} \xi)}{\e^{2/3}} - \frac{\xi}{\e^{4/3}}.
\end{equation}

It follows from~\eqref{eq:chi_definition} and the definition of $f^\e_\aut$ that
\begin{equation}\label{eq:full_expansion}
	f_\aut^\e(x,y,t) = y - t + \e^{2/3} \chi_\aut^\e(x, \e^{2/3} y).
\end{equation}
As a consequence, we need only understand the convergence of $\chi_\aut^\e$ as $\e\to0$ to conclude the proof of \Cref{prop:general}.

\subsubsection{Step (iii): the reduction to the case where $W^\e$ converges in probability}\label{sec:reduction}

We now reduce to the case where the random advection converges in probability instead of simply in distribution.  For this, we need the following lemma.
\begin{lemma}\label{lem:convergence}
	Suppose that~\Cref{assumption:advection} holds, and assume that $W^\e$ converges in probability to a standard Brownian motion $W$.  Let $\chi_\aut^\e$ be given by~\eqref{eq:chi_definition} with $\rho^\e$ solving~\eqref{eq:general_metric_problem}.  There exists $\Omega' \subset \Omega$ with $P(\Omega') = 1$ such that, for every $\omega \in \Omega'$, $\chi_\aut^\e(\cdot,\cdot,\omega)$ converges locally uniformly in $\R^{n-1}\times[0,\infty)$ to the solution $\chi$ of~\eqref{eq:corrector} when $\beta > 2/3$ and to the solution $\chi_{\rm visc}$ of~\eqref{eq:corrector_viscous} when $\beta = 2/3$.

\end{lemma}
The lemma is proved in the next subsection.  On the face of it, \Cref{lem:convergence} requires stronger assumptions than~\Cref{prop:general}.  We now show how to get around this.

\begin{proof}[Proof of \Cref{prop:general} using \Cref{lem:convergence}]

Fix any sequence $\e_n \to 0$. It follows from~\cite[Theorem~4.6, Chapter 1]{IkedaWatanabe} that there exists a subsequence $\e_{n_k}\to 0$, a probability space $(\widehat \Omega, \widehat \cF,\widehat P)$, and processes $\widehat W^{\e_{n_k}}$ and $\widehat W$ defined on $(\widehat \Omega, \widehat \cF, \widehat P)$ such that
\begin{equation}\label{eq:hat_W}
	\widehat W \stackrel{d}{=} W, \quad
		\widehat W^{\e_{n_k}} \stackrel{d}{=} W^{\e_{n_k}}
			\quad \text{ and } \quad
		\widehat W^{\e_{n_k}} \toas \widehat W~ \text{ as $k\to\infty$.}
\end{equation}

\medskip

Let $\widehat \rho_k$ be the unique solution of~\eqref{eq:general_metric_problem}  given by \Cref{lem:rho} with $w$ replaced by
\[
	\widehat w_k(y) := \sigma \e_{n_k}^{1/3} \widehat W^{\e_{n_k}}_y(\e_{n_k}^{2/3} y).
\]
and, for all $(x,\xi) \in \R$, set
\[
	\widehat \chi_k(x,\xi) = \frac{\widehat \rho_k(x, \e_{n_k}^{-2/3} \xi)}{\e_{n_k}^{2/3}} - \frac{\xi}{\e_{n_k}^{4/3}}.
\]

\medskip

We consider the case $\beta > 2/3$.  \Cref{lem:convergence} yields that $\widehat \chi_k$ converges almost surely, and thus in distribution, to $\chi$.  From the well-posedness of~\eqref{eq:geometric} and the fact that $W^{\e_{n_k}} \stackrel{d}{=} \widehat W^{\e_{n_k}}$, it follows that $\widehat \chi_k \stackrel{d}{=} \chi^{\e_{n_k}}_\aut$, and thus, $\chi^{\e_{n_k}}_\aut \tod \chi$.  Since this holds for every sub-sequence $\e_k$, it follows that $\chi_\aut^\e \tod \chi$.

\medskip

When $\beta = 2/3$, the argument is similar; hence, we omit it.
\end{proof}

\subsubsection{Step (iv): the proof of \Cref{lem:convergence} using the half-relaxed limits}\label{sec:convergence}

We now prove, under the slightly stronger assumptions on the convergence of $W^\e$ to $W$, that $\chi_\aut^\e$ converges to $\chi$, if $\beta>2/3$, and to $\chi_{\rm visc}$, if $\beta=2/3$,.

\medskip

%
%

Consider the nonlinear error function $N_\e: \R^{n-1}\times \R \to \R$ given by
\[
	N_\e(p, s) := \frac{1}{r\e^{4/3}}\left(1 + r \e^{4/3} s + \frac{r \e^{4/3}}{2} |p|^2 - \left(1 + 2\e^{4/3} s +  \e^{4/3} |p|^2 +  \e^{8/3} s^2 \right)^{r/2} \right).
\]
and observe that, in the limits $\e^{4/3}s,\, \e^{4/3} |p|^2\to 0$,
\begin{equation}\label{eq:N_asymptotics}
	N_\e(p,s) = O \left(\e^{4/3} s^2\right) + O\left( \e^{4/3} |p|^4\right).
\end{equation}
Using~\eqref{eq:general_metric_problem} and~\eqref{eq:chi_definition}, we formally see that, for any $(x,\xi)\in\R^{n-1}\times\R_+$, $\chi_\aut^\e$ satisfies
\begin{equation}\label{eq:chi_equation}
\begin{split}
	&\chi_{\aut,\xi}^\e + \frac{1}{2}|\nabla_x \chi_\aut^\e|^2
		+ \epsilon^{-1/3} u_{\aut,\parallel}w(\eps^{-2/3}\cdot)
		- \e^{\beta-2/3} \Delta_x \chi_\aut^\e
		= N_\e(\nabla_x \chi_\aut^\e, \chi_{\aut,\xi}^\e)\\
		 &\quad	 - \epsilon^{1/3} u_{\aut,\perp}(\cdot,\eps^{-2/3} \cdot)\cdot \nabla_x \chi_\aut^\e
		- \epsilon u_{\parallel} w(\eps^{-2/3} \cdot) \chi_{\aut,\xi}^\e
			+  \e^{\beta+2/3} \chi^\e_{\aut,\xi\xi},
\end{split}
\end{equation}
where $u_{\aut, \parallel}$ and $u_{\aut, \perp}$ are defined in an analogous manner as $u_\aut$. 
We now justify this formal computation.  First we show that $\chi_\aut^\e$ is a viscosity super-solution of~\eqref{eq:chi_equation}.  Fix $(x_0,\xi_0) \in \R^{n-1}\times\R_+$ and a test function $\psi$ such that $\chi_\aut^\e - \psi$ has a local minimum at $(x_0,\xi_0)$ and let
\[
	\overline \psi(x, y) = y + \e^{2/3} \psi(x, \e^{2/3} y).
\]
It follows from the definition of $\chi^\e_\aut$ in~\eqref{eq:chi_definition} that $\rho^\e - \overline\psi$ has a local minimum at $(x_0, \xi_0 \e^{-2/3})$.  Thus, at $(x_0,\xi_0 \e^{-2/3})$,
\[\begin{split}
	- r \e^{2/3 + \beta} &(\Delta_x \psi + \e^{4/3} \psi_{\xi\xi}) + r \e u_{\aut} \cdot (\nabla_x \psi, 1 + \e^{4/3} \psi_\xi) + |(\nabla_x \psi, 1 + \e^{4/3} \psi_\xi)|^r\\
		&= - r\frac{\e^\beta}{2} \Delta \overline \psi + r \e ut \cdot \nabla \overline \psi + |\nabla \overline \psi|^r
		\geq 1.
\end{split}\]
Dividing by $r \e^{2/3}$ and rearranging yields
\[\begin{split}
	&\psi_\xi + \frac{1}{2}|\nabla_x \psi|^2
		+ \epsilon^{-1/3} u_{\aut, \parallel}w(\eps^{-2/3}\cdot)
		- \e^{\beta -2/3} \Delta_x \psi\\
		&\quad\geq N_\e(\nabla_x \psi,\psi_\xi)
		 	 - \epsilon^{1/3} u_{\aut, \perp}(\cdot,\eps^{-2/3} \cdot)\cdot \nabla_x \psi
			- \epsilon u_{\aut, \parallel} w(\eps^{-2/3} \cdot) \psi_\xi
			+ \e^{\beta +2/3} \psi_{\xi\xi}.
\end{split}\]
A similar argument shows that $\chi_\aut^\e$ is a sub-solution of~\eqref{eq:chi_equation}.  

\medskip

In order to work with stochastic viscosity solutions in the limit, we set
\begin{equation}\label{eq:ochi_definition}
	\ochi_\aut^\e(x,\xi)
		:= \chi_\aut^\e(x,\xi) + u_{\aut, \parallel}(x) W^\e(\xi),
\end{equation}
and, in view of~\eqref{eq:chi_definition},~\eqref{eq:ochi_definition}, and the bounds in~\Cref{lem:rho}, observe that
\begin{equation}\label{eq:chi_bound}
	\left|\ochi_\aut^\e(x,\xi)\right|
		\leq \mu_1 |\xi|
			+ \frac{\mu_2}{2} \xi^2
			+ \mu_3 \int_0^\xi |W^\e(\xi')|^2d\xi',
\end{equation}
a bound that is crucial to take the half-relaxed limits of $\chi_\aut^\e$.

\medskip

It follows from~\eqref{eq:chi_equation} that, at any point $(x,\xi) \in \R^{n-1}\times\R_+$,
\begin{equation}\label{eq:ochi_equation}
\begin{split}
	&\ochi_{\aut,\xi}^\e
			+ \frac{1}{2}|\nabla_x \ochi_\aut^\e - W^\e(\xi) \nabla_x u_{\aut\parallel}|^2
			- \e^{\beta - 2/3} \Delta_x \ochi_\aut^\e
			+ \e^{\beta - 2/3} \Delta_x u_{\aut, \parallel} W^\e(\xi)\\
		&\quad = N_\e(\nabla_x \ochi_\aut^\e - W^\e(\xi) \nabla_x u_{\aut,\parallel},\ochi_{\aut,\xi}^\e - \e^{-1/3} u_{\aut, \parallel} w(\e^{-2/3}\xi))\\
&\qquad		- \epsilon^{1/3} u_{\aut, \perp}(x,\eps^{-2/3} \xi)\cdot (\nabla_x \ochi_\aut^\e - W^\e(\xi) \nabla_x u_{\aut, \parallel})\\
&\qquad		- \epsilon u_{\aut, \parallel}w(\eps^{-2/3} \xi) ( \ochi_{\aut,\xi}^\e - \e^{-1/3}u_{\parallel} w(\e^{-2/3} \xi))
				+ \e^{\beta + 2/3} (\ochi_{\aut,\xi\xi}^\e - \e^{-1}w_y( \e^{-2/3} \xi),
\end{split}
\end{equation}
where we used that $W_\xi^\e(\xi) = \e^{-1/3} w(\e^{-2/3} \xi)$ and $W_{\xi\xi}^\e(\xi) = \e^{-1} w_y(\e^{-2/3}\xi)$.

\medskip

Furthermore, \eqref{eq:chi_bound} yields that $\ochi^\e_\aut$ is locally bounded with probability one.  Indeed, let $\Omega' \subset \Omega$ be such that $P(\Omega') = 1$ and, for all $\omega \in \Omega'$, $W(\cdot,\omega)$ is continuous and $W^\e(\cdot,\omega)$ converges to $W(\cdot,\omega)$ locally uniformly. Then $W^\e$ is locally bounded as well.  The bound on $\ochi^\e_\aut$ follows.

\medskip

As a result, for any $\omega \in \Omega'$, the classical half-relaxed limits
\begin{equation}\label{eq:half_relaxed_G}
	\ochi^*(x,\xi, \omega)
		:= \limsup_{\substack{(x',\xi') \to (x,\xi),\\ \e\to0}} \ochi_\aut^\e(x',\xi',\omega)
	\quad \text{ and }\quad
	\ochi_*(x,\xi,\omega)
		:= \liminf_{\substack{(x',\xi') \to (x,\xi),\\ \e\to0}}\ochi_\aut^\e(x',\xi',\omega),
\end{equation}
are well-defined.  By construction, $\ochi_* \leq \ochi^*$.  The key step to proving the opposite inequality is to show that these are sub- and super-solutions of the same equation.

\begin{lemma}\label{lem:half_relaxed_G}
	For each $\omega \in \Omega'$, the half-relaxed limits $\ochi^*(\cdot,\cdot,\omega)$ and $\ochi_*(\cdot,\cdot,\omega)$ satisfy repectively
	\begin{equation}\label{eq:ochiG^*}
		\begin{cases}
			\ochi_\xi^* + \frac{1}{2}|\nabla_x \ochi^* - W \nabla_x u_{\aut, \parallel}|^2 - \delta_{\frac23\beta} \Delta_x(\ochi^* - W u_{\aut, \parallel}) \leq 0 \qquad &\text{ in }~~ \R^{n-1}\times\R_+,\\
			\ochi^* = 0 &\text{ on }~~ \R^{n-1} \times\{0\},
		\end{cases}
	\end{equation}
	and
	\begin{equation}\label{eq:ochiG_*}
		\begin{cases}
			\ochi_{*,\xi} + \frac{1}{2}|\nabla_x \ochi_* - W(\xi) \nabla_x u_{\aut, \parallel}|^2 - \delta_{\frac23\beta} \Delta_x(\ochi_* - W u_{\aut, \parallel}) \geq 0 \qquad &\text{ in }~~ \R^{n-1}\times\R_+,\\
			\ochi_* = 0 &\text{ on }~~ \R^{n-1} \times\{0\}.
		\end{cases}
	\end{equation}
\end{lemma}
\begin{proof}
	Since the proofs are similar, we only show the argument for~\eqref{eq:ochiG^*}.  In what follows we work with fixed $\omega \in \Omega'$ and, hence, suppress it for notational simplicity.
	
	\medskip
	
	We begin with the behavior of $\ochi^*$ at $\xi = 0$.  For this, we note that~\eqref{eq:chi_bound}, the continuity of $W$, and the convergence of $W^\e$ to $W$ imply that $\ochi^* = 0$ on $\R^{n-1}\times\{0\}$.
	
	\medskip
	
	Next assume that, for some test function $\psi$, $\ochi^* - \psi$ has a strict local maximum at $(x_0,\xi_0)\in \R^{n-1}\times\R_+$.  It follows from the definition of $\ochi^*$ that there exist sequences $(x_k,\xi_k) \in \R^{n-1}\times\R_+$ and $\epsilon_k > 0$ such that $\ochi^{\e_k}_\aut - \psi$ has a local maximum at $(x_k,\xi_k)$ and, as $k \to \infty$, $\e_k \to 0$, $(x_k, \xi_k) \to (x_0,\xi_0)$, and $\ochi^{\e_k}(x_k,\xi_k) - \psi(x_k,\xi_k)\to \ochi^*(x_0,\xi_0) - \psi(x_0,\xi_0)$.
	
	\medskip
	
	Using~\eqref{eq:ochi_equation}, we find that, at $(x_k,\xi_k)$,
\[\begin{split}
	&\psi^{\e_k}_\xi
			+ \frac{1}{2}|\nabla_x \psi^{\e_k} - W^{\e_k} \nabla_x u_{\aut, \parallel}|^2
			- \e_k^{\beta - 2/3} \Delta_x \psi^{\e_k}
			+ \e_k^{\beta - 2/3} \Delta_x u_{\aut, \parallel} W^{\e_k}\\
		&\quad \geq N_{\e_k}(\nabla_x \psi^{\e_k} - W^{\e_k} \nabla_x u_{\aut, \parallel},\psi_\xi^{\e_k} - \e_k^{-1/3} u_{\aut, \parallel} w(\e_k^{-2/3}\xi_k))\\
&\qquad		- {\e_k}^{1/3} u_{\aut, \perp}(x_k,\eps_k^{-2/3} \xi_k)\cdot (\nabla_x \psi^{\e_k} - W^{\e_k} \nabla_x u_{\aut, \parallel})\\
&\qquad		- {\e_k} u_{\aut, \parallel} w(\eps_k^{-2/3} \xi_k) ( \psi_\xi^{\e_k} - \e_k^{-1/3}u_{\aut, \parallel} w(\e_k^{-2/3} \xi_k))
				+ \e_k^{\beta + 2/3} (\psi_{\xi\xi}^{\e_k} - \e_k^{-1} u_{\aut, \parallel} w_y( \e_k^{-2/3} \xi_k)).
\end{split}\]
By assumption, we have that $W^{\e_k}(\xi_k)\to W(\xi_0)$.  Hence, the last two terms on the left hand side tend to zero if $\beta > 2/3$ and to $- \Delta_x (\psi - u_{\aut, \parallel} W)$ if $\beta = 2/3$.  In addition, it is clear that $W^{\e_k}(\xi_k) \nabla_x u_{\aut, \parallel}(x_k)$ converges, as $k\to\infty$, to $W(\xi_0) \nabla_x u_{\aut, \parallel}(x_0)$.

\medskip

The second, third, and fourth terms on the right hand side clearly tend to zero as $k\to\infty$, while the first term also does due to~\eqref{eq:N_asymptotics}.

\medskip

Thus, letting $k\to\infty$, we find that, at $(x_0,\xi_0)$,
	\[\begin{split}
		&\psi_\xi + \frac{1}{2}|\nabla_x \psi - W \nabla_x u_{\aut, \parallel}|^2 - \delta_{\frac23\beta}\Delta_x \left( \psi - u_{\aut, \parallel}W\right)
		\geq 0.
	\end{split}\]
\end{proof}

We now combine the above results to prove \Cref{lem:convergence}.

\begin{proof}[Proof of \Cref{lem:convergence}]
	Since the two claims are proved similarly, we only include the details for the first.  Moreover, we again fix $\omega \in \Omega'$ throughout but omit this dependence to simplify the notation.
	
	\medskip

	It follows from the comparison principle and \Cref{lem:half_relaxed_G} that $\ochi^*\leq \ochi_*$ on $\R^{n-1}\times[0,\infty)$, while, as noted before, $\ochi_* \leq \ochi^*$.  We conclude that $\ochi^* = \ochi_*$ and denote this function $\ochi$.  This equality and the definition of the half-relaxed limits~\eqref{eq:half_relaxed_G}, yields that, as $\e\to0$, $\ochi_\aut^\e$ converges to $\ochi$ locally uniformly in $\R^{n-1}\times[0,\infty)$.
	\medskip
	
	It follows from \Cref{lem:half_relaxed_G} and the fact that $\ochi^* = \ochi_* = \ochi$, that $\ochi - u_{\aut, \parallel} W$ solves~\eqref{eq:corrector}.  Uniqueness thus gives that $\chi = \ochi - u_{\aut, \parallel} W$.  Furthermore, the convergences of $W^\e$ to $W$ and $\ochi_\aut^\e$ to $\ochi$ and the definition of $\ochi_\aut^\e$ give that $\chi_\aut^\e$ converges, as $\e\to0$, locally uniformly to $\chi$.  This concludes the proof.
	
%
	
\end{proof}

\subsection{The non-autonomous case: $1 \leq \alpha < \infty$}

Arguing as in \Cref{sec:reduction}, we assume without loss of generality that, as $\e\to0$, $W^\e$ converges to $W$ in probability.  We fix $\Omega' \subset \Omega$ to be the set of full probability such that $W$ is continuous and $W^\e$ converges locally uniformly to $W$ as used in \Cref{sec:convergence}.

\medskip

We again work in the more general framework.  \Cref{thm:geometric_ti,thm:FKPP} reduce to the following result.

\begin{prop}\label{prop:general_td}
	Suppose that \Cref{assumption:advection} holds, $\alpha \geq 1$, $\beta \geq 2/3$, $r \in[1,2]$, and $\omega \in \Omega'$, and let $f^\e$ solve
	\begin{equation}\label{eq:general_problem_td}
	\begin{cases}
		f_t^\e + \e u \cdot \nabla f^\e + \frac{1}{r} |\nabla f^\e|^r + \frac{r-1}{r} = \frac{\e^\beta}{2} \Delta_x f^\e
			\qquad &\text{ in }~~ \R^n\times\R_+,\\
		f^\e = f^\e_\aut
			 &\text{ on }~~ \R^{n}\times \{0\}.
	\end{cases}
	\end{equation}
	Then, as $\e\to0$ and locally uniformly on $\R^{n-1}\times \{(\xi,\tau) \in \R \times[0,\infty): \xi \geq \tau\}$,
	\begin{equation}\label{eq:corrector_td}
		\chi^\e(x,\xi,\tau)
			:= \frac{1}{\e^{2/3}}  f^\e\left(x, \frac{\xi}{\e^{2/3}}, \frac{\tau}{\e^{2/3}}\right) - \frac{1}{\e^{4/3}} (\xi - \tau).
\end{equation}
	converges to the unique solution~$\chi$ of~\eqref{eq:corrector} when $\beta >2/3$ and to the unique solution~$\chi_{\rm visc}$ of~\eqref{eq:corrector_viscous} when $\beta=2/3$.
%
\end{prop}


\subsubsection{A priori bounds on $f^\e$}\label{sec:a_priori_td}

\begin{lemma}\label{lem:a_priori_td}
	There exists $C>0$, which is independent of $\e$, such that, for all $(x,y)\in \R^n$,
\[
	|f^\e(x,y,t) - f^\e_\aut(x,y,t)| \leq C\|u\|_{C^1} \e^{1+\alpha} t^2.
\]
\end{lemma}
\begin{proof}
Let $\rho^\e$ be the solution of~\eqref{eq:general_metric_problem}.  It follows from \Cref{lem:rho} that $\|\nabla \rho^\e\|_{\infty} \leq  C_L$, for some $C_L>0$ that does not depend on $\e$.  Recalling that $f^\e_\aut = \rho^\e - t$, we find $\|D f^\e_\aut\|_\infty \leq C_L$.

\medskip

To prove the claim, we show that $\overline f^\e(x,y,t) := f^\e_\aut(x,y,t) + C_L \|u\|_{C^1}\e^{1+\alpha} t^2$ and $\underline f^\e(x,y,t) := f^\e_\aut(x,y,t) - C_L \|u\|_{C^1}\e^{1+\alpha} t^2$ are, respectively, super- and sub-solutions of~\eqref{eq:general_problem_td}.  Once this is established, the claim follows by a standard application of the comparison principle.  The proofs are similar so we only show the upper bound.

\medskip

A straightforward computation and an application of Taylor's theorem yield
\[\begin{split}
	\overline f_t^\e + \e u \cdot \nabla \overline f^\e + \frac{1}{r} |\nabla \overline f^\e|^r + \frac{r-1}{r} - \frac{\e^\beta}{2} \Delta_x \overline f^\e
		&= \e (u - u_\aut) \cdot \nabla f_\aut^\e + 2 C_L \|u\|_{C^1} \e^{1+\alpha}  t\\
		&\geq - \e (\|u\|_{C^1} \e^\alpha t) \|\nabla f_\aut^\e\|_\infty + 2 C_L \|u\|_{C^1} \e^{1+\alpha} t
		\geq 0,
\end{split}\]
that is, $\overline f^\e$ is a super-solution of \eqref{eq:general_problem_td}, as claimed.
\end{proof}

At this point, we are able to conclude the proof in the case where $\alpha > 1$.
\begin{proof}[Proof of \Cref{prop:general_td} for $\alpha > 1$]
	Combining the definition of $\chi^\e$ with the estimates of \Cref{lem:a_priori_td}, we find $C>0$, which is independent of $\e$, such that, for all $(x,\xi,\tau) \in \R^n \times\R_+$,
	\[
		|\chi^\e(x,\xi,\tau) - \chi^\e_\aut(x,\xi)|
			\leq C \e^{\alpha - 1} \tau^2.
	\]
	Notice that $\alpha - 1 > 0$.  The result then follows from \Cref{prop:general}, which yields the convergence of $\chi^\e_\aut$ to $\chi$ if $\beta>2/3$ and to $\chi_{\rm visc}$ if $\beta = 2/3$.
\end{proof}


\subsubsection{The half-relaxed limits when $\alpha = 1$}

 First, in anticipation of the limiting equation, we introduce
\begin{equation}\label{eq:ochi_td}
	\overline \chi^\e(x,\xi,\tau)
		:= \chi^\e(x,\xi,\tau)  + u_{\parallel}(x, \e^{1/3} \tau)W^\e(\xi).
\end{equation}
Arguing as for~\eqref{eq:ochi_equation}, we find
\begin{equation}\label{eq:ochi_equation_td}
\begin{split}
	&\ochi_\tau^\e + \ochi_\xi^\e
			+ \frac{1}{2}|\nabla_x \ochi^\e - W^\e \nabla_x u_{\parallel}|^2
			- \e^{\beta - 2/3} \Delta_x \ochi^\e
			+ \e^{\beta - 2/3} \Delta_x u_{\parallel} W^\e\\
		&\quad = N_\e(\nabla_x \ochi^\e - W^\e \nabla_x u_{\parallel},\ochi_\xi^\e - \e^{-1/3} u_{\parallel} w(\e^{-2/3}\cdot))\\
&\qquad		
				- \epsilon^{1/3} u_{\perp}(\cdot,\eps^{-2/3} \cdot)\cdot (\nabla_x \ochi^\e - W^\e \nabla_x u_{\parallel})
		- \epsilon u_{\parallel}w(\eps^{-2/3} \cdot) (\ochi_\xi^\e - \e^{-1/3}u_{ \parallel} w(\e^{-2/3} \cdot))\\
		&\qquad
				+ \e^{\beta + 2/3} (\ochi_{\xi\xi}^\e - \e^{-1} w_y( \e^{-2/3} \cdot)
				+ \e^{1/3}u_{\parallel,t} W^\e.
\end{split}
\end{equation}
Notice that~\eqref{eq:ochi_equation_td} is the same as~\eqref{eq:ochi_equation} except for the additional time derivative of $\ochi^\e$ on the left, the last term on the right, and the fact that $u$ is dependent on $t$.

\medskip

It follows from \Cref{lem:a_priori_td} that there exists $C>0$, which is independent of $\e$, such that, for every $(x,y,t)\in \R^n\times \R_+$,
\begin{equation}\label{eq:initial_datum_td}
	|\ochi^\e(x,\xi,\tau) - \ochi^\e_\aut(x,\xi,\tau)| \leq C \tau^2.
\end{equation}
Combining this with~\eqref{eq:chi_bound}, we find that $\ochi^\e$ is locally bounded in $\R^n\times \R_+$.  Thus, the half-relaxed limits
\begin{equation}\label{eq:half_relaxed_td}
	\ochi^*(x,\xi,\tau)
		:= \limsup_{\substack{(x',\xi',\tau') \to (x,\xi,\tau),\\ \e\to0}} \ochi^\e(x',\xi',\tau')
	\quad \text{and}\quad
	\ochi_*(x,\xi,\tau)
		:= \liminf_{\substack{(x',\xi',\tau') \to (x,\xi,\tau),\\ \e\to0}}\ochi^\e(x',\xi',\tau')
\end{equation}
are well-defined.

\medskip

Again, arguing as in the proof of \Cref{lem:half_relaxed_G}, we obtain the following result.
\begin{lemma}\label{lem:half_relaxed_td}
	For $\omega \in \Omega'$, the half-relaxed limits $\ochi^*(\cdot,\cdot,\omega)$ and $\ochi_*(\cdot,\cdot,\omega)$ satisfy, repectively
	\begin{equation}\label{eq:ochiG^*_td}
		\begin{cases}
			\ochi_\tau^* + \ochi_\xi^* + \frac{1}{2}|\nabla_x \ochi^* - W \nabla_x u_{\aut,\parallel}|^2 - \delta_{\beta \frac23} \Delta_x(\overline \chi^* - W u_{\aut,\parallel}) \leq 0  \qquad&\text{in } ~~\R^{n}\times\R_+,\\
			\ochi^* = \chi_\beta + u_{\aut,\parallel} W
				&\hspace{-.78in}\text{on } ~~ \R^{n-1}\times [0,\infty) \times\{0\},
		\end{cases}
	\end{equation}
	and
	\begin{equation}\label{eq:ochiG_*_td}
		\begin{cases}
			\ochi_{*,\tau} + \ochi_{*,\xi} + \frac{1}{2}|\nabla_x \ochi_* - W \nabla_x u_{\aut,\parallel}|^2 - \delta_{\beta \frac23} \Delta_x(\overline \chi^* - W u_{\aut,\parallel}) \geq 0  \qquad&\text{in } ~~ \R^n\times\R_+,\\
			\ochi_* = \chi_\beta + u_{\aut, \parallel} W &\hspace{-.78in}\text{on } ~~\R^{n-1} \times [0,\infty) \times\{0\},
		\end{cases}
	\end{equation}
where $\chi_\beta$ is $\chi$ when $\beta > 2/3$ and $\chi_{\rm visc}$ when $\beta = 2/3$.
\end{lemma}
\begin{proof}
The only difference between the proof of~\eqref{eq:ochiG^*_td} and~\eqref{eq:ochiG_*_td} and that of the analogous claims in \Cref{lem:half_relaxed_G} is about the initial data. This is, however, handled using~\eqref{eq:initial_datum_td} and \Cref{prop:general}, which gives the convergence of $\chi^\e_\aut + u_{\aut,\parallel} W^\e$ to $\chi + u_{\aut,\parallel}W^\e$ and $\chi_{\rm visc} + u_{\aut,\parallel}W^\e$ when $\beta > 2/3$ and $\beta = 2/3$ respectively.  We omit the rest of the details.

%
\end{proof}

\subsubsection{The proof of \Cref{prop:general_td} when $\alpha=1$}

 We now finish the proof of \Cref{prop:general_td} when $\alpha = 1$.  Recall the case when $\alpha >1$ was dealt with in \Cref{sec:a_priori_td}.
 
The natural way to proceed is to use the comparison principle, as above, to conclude that $\ochi^* = \ochi_*$.  While~\eqref{eq:ochiG^*_td} and~\eqref{eq:ochiG_*_td} enjoy the comparison principle, we do not have any ordering of $\overline \chi_*$ and $\overline\chi^*$ when $\xi < 0$ and, thus, cannot immediately apply comparison.  To overcome this, we apply a simple transformation that allows to use the comparison principle along rays where $\xi - \tau$ is constant.
 
\begin{proof}[Proof of \Cref{prop:general_td} when $\alpha = 1$]
	Throughout this proof, we fix $\omega \in \Omega'$ and suppress the dependence on $\omega$.
		
	\medskip
	
 We first show that, for any fixed $\xi_0 \geq 0$, $\ochi^* = \ochi_*$ on $\R^{n-1}\times R_{\xi_0}$, where $R_{\xi_0} := \{(\xi,\tau) \in \R \times [0,\infty): \xi - \tau = \xi_0\}$.  
 
 \medskip
 
 Let
	\begin{equation}\label{eq:chiChi}
		\Chi^*(x,\zeta, \tau) := \ochi^*(x, \zeta + \tau, \tau),~~
		\Chi_*(x,\zeta,\tau) := \ochi_*(x,\zeta + \tau, \tau),
			~~\text{and}~~
		\mathcal{W}(\zeta,\tau) = W(\zeta + \tau).
	\end{equation}
	We claim that
	\begin{equation}\label{eq:Chi^*}
		\begin{cases}
			\Chi_\tau^* + \frac{1}{2}|\nabla_x \Chi^* - \cW \nabla_x u_{\aut,\parallel}|^2 - \delta_{\beta \frac23} \Delta_x(\Chi^* - \cW u_{\aut,\parallel}) \leq 0  \quad&\text{in } ~\R^{n-1}\times \{\xi_0\} \times\R_+,\\
			\Chi^* = \chi_\beta + u_{\aut,\parallel} \cW
				&\text{on } ~ \R^{n-1}\times \{\xi_0\} \times\{0\},
		\end{cases}
	\end{equation}
	and
	\begin{equation}\label{eq:Chi_*}
		\begin{cases}
			\Chi_{*,\tau} + \frac{1}{2}|\nabla_x \Chi_* - \cW \nabla_x u_{\aut,\parallel}|^2 - \delta_{\beta \frac23} \Delta_x(\Chi^* - \cW u_{\aut,\parallel}) \geq 0  ~~&\text{in } ~ \R^{n-1}\times\{\xi_0\}\times\R_+,\\
			\Chi_* = \chi_\beta + u_{\aut, \parallel} \cW &\text{on } ~\R^{n-1} \times \{\xi_0\} \times\{0\}.
		\end{cases}
	\end{equation}
	The proofs of~\eqref{eq:Chi^*} and~\eqref{eq:Chi_*} are similar so we omit the one for~\eqref{eq:Chi_*}.  Assume that, for some test function $\Psi$, $\Chi^*(\cdot,\xi_0,\cdot) - \Psi$ has a strict local maximum at $(x_0, \tau_0) \in \R^{n-1}\times \R_+$.  For any $\theta >0$, let
	\[
		\Psi_\theta(x,\zeta, \tau) := \Psi(x,\tau) + \frac{1}{\theta}( \zeta - \xi_0)^4.
	\]
	Due to~\eqref{eq:bounds_rho1}, if $\theta$ is sufficiently small, then there exists a local maximum of $\Chi^* - \Psi_\theta$ at some point $(x_\theta, \zeta_\theta, t_\theta)$, and, furthermore, as $\theta\to 0$, $(x_\theta,\zeta_\theta,t_\theta) \to (x_0, \xi_0,t_0)$.
	
	\medskip
	
	Let $\psi_\theta(x,\xi,\tau) = \Psi_\theta(x, \xi - \tau, \tau)$.   It follows from the definition of $\Chi^*$ and the choice of $(x_\theta, \zeta_\theta, \tau_\theta)$ that $\ochi^* - \psi_\theta$ has a local maximum at $(x_\theta, \zeta_\theta + \tau_\theta, \tau_\theta)$.  Due to~\eqref{eq:ochiG^*}, we find, at $(x_\theta, \zeta_\theta + \tau_\theta, \tau_\theta)$,
	\[
		\psi_{\theta,\tau} + \psi_{\theta,\xi} + \frac{1}{2} |D_x \psi_\theta - W D_x u_{\aut, \parallel}|^2 - \delta_{\beta \frac23}(\psi_\theta - W u_{\aut,\parallel}) \leq 0.
	\]
	This implies that, at $(x_\theta, \zeta_\theta, \tau_\theta)$,
	\[\begin{split}
		0 &\geq \Psi_{\theta,\tau} + \frac{1}{2} |D_x \Psi_\theta - \cW D_x u_{\aut, \parallel}|^2 - \delta_{\beta \frac23}(\Psi_\theta - \cW u_{\aut,\parallel})\\
			&= \Psi_{\tau} + \frac{1}{2} |D_x \Psi - \cW D_x u_{\aut, \parallel}|^2 - \delta_{\beta \frac23}(\Psi - \cW u_{\aut,\parallel}),
	\end{split}\]
	where we used the relationships between $\psi_\theta$, $\Psi_\theta$, and $\Psi$, as well as the relationship between $W$ and $\cW$.  We conclude that~\eqref{eq:Chi^*} holds by letting $\theta\to0$.
	
	\medskip
	
	Due to~\eqref{eq:Chi^*} and~\eqref{eq:Chi_*} and the fact that $\Chi^*(x, \xi_0, 0) = \chi(x,\xi_0) = \Chi_*(x,\xi_0,0)$ for all $x\in \R^{n-1}$, the comparison principle implies that $\Chi^* \leq \Chi_*$ in $\R^{n-1} \times \{\xi_0\} \times \R_+$.  Hence, by~\eqref{eq:chiChi}, $\ochi^* \leq \ochi_*$ on $\R^{n-1}\times R_{\xi_0}$.
	
	\medskip
	
	On the other hand, we have $\ochi_* \leq \ochi^*$ by construction.  Thus, $\ochi^* = \ochi_*$ on $\R^{n-1}\times R_{\xi_0}$.  
	
	\medskip
	
	Moreover, since $\chi_\beta + u_{\aut,\parallel} W$ satisfies both~\eqref{eq:ochiG^*} and~\eqref{eq:ochiG_*} on $\R^{n-1}\times\R_+ \times \R_+$ similar arguments show that $\ochi_* = \ochi^* = \chi + u_{\aut,\parallel} W$ on $\R^{n-1}\times R_{\xi_0}$.
	
	\medskip
	
	This holds for all $\xi_0 \geq 0$.  As a result, $\ochi^* = \ochi_* = \chi_\beta + u_{\aut,\parallel} W$ on $\R^{n-1}\times \{(\xi,\tau) : \xi  \geq \tau \geq 0\}$, which implies that $\ochi^\e$ converges locally uniformly on $\R^{n-1}\times \{(\xi,\tau) : \xi  \geq \tau \geq 0\}$ to $\chi_\beta + u_{\aut,\parallel} W$.    The proof is finished by noting that the locally uniform convergence of $\chi^\e$ to $\chi_\beta$ follows from the combination of this and the convergence of $W^\e$ to $W$.

\end{proof}

\section{Front asymptotics for the initial value problem: the G-equation}\label{sec:G_ivp}

We show that the asymptotics for the front of the perturbed traveling wave solutions $G^\e_\ptw$ yield the asymptotics for solutions with more general initial datum; that is, we prove \Cref{prop:geometric_td}.

\begin{proof}[Proof of \Cref{prop:geometric_td}]
With $\underline G$ and $\overline G$ as in \Cref{assumption:initial_datum}, let $G^\e_\ptw$ be the solution constructed in \Cref{thm:geometric_ti}.  The goal is to create sub- and super-solutions using these functions.

\medskip

Fix $\delta \in (0,1/2$ and let $\underline \phi_\delta \in C^1(\R)$ be an approximation of
\[
	\underline \phi(y) := \begin{cases}
				\underline G(y/2), \qquad &\text{ if } ~~ y \geq 0,\\
				\underline G(2y), &\text{ if } ~~ y < 0,
			\end{cases}
\]
such that $\underline \phi_\delta = \underline \phi$ on $\R\times(-\delta,\delta)$ and $\underline \phi_\delta'>0$ in $\R$.  Furthermore, we may assume that $\|\underline \phi_\delta\|_{C^{0,1}(-1,1)} \leq 2 \|\phi\|_{C^{0,1}(-1,1)}$.

\medskip

Let $C_\phi = 8 \|\underline \phi\|_{C^{0,1}(-1,1)}$, notice that $C_\phi \geq 4 \|\underline \phi_\delta\|_{C^{0,1}(-1,1)}$, and define
\[
	\underline \mu_\delta := \underline \phi_\delta \circ G_\ptw^\e - 2\|\phi\|_{C^{0,1}(-1,1)} \delta
	\qquad \text{ and } \qquad
	\underline \mu := \underline \phi \circ G_\ptw^\e.
\]
It is immediate that $\{\underline\mu \leq 0\} = \{G_\ptw^\e \leq 0\}$ and $\{ \underline \mu = 0\} = \{G_\ptw^\e = 0\}$.

\medskip

We show that $\underline \mu_\delta$ is a sub-solution of~\eqref{eq:geometric}.  Indeed, fix any test function $\psi$ and any point $(x_0,y_0,t_0) \in \R^{n-1}\times \R\times \R_+$ such that $\underline \mu_\delta - \psi$ has a strict local maximum at $(x_0,y_0,t_0)$.  Since $\underline \phi_\delta$ is strictly increasing, it follows that $G_\ptw^\e - \underline \phi_\delta^{-1}\circ \psi$ has a strict local maximum at $(x_0,y_0,t_0)$.  Since $\underline \phi_\delta^{-1}$ is $C^1$, $\underline \phi_\delta^{-1}\circ \psi$ is a valid test function and, hence, we find that, at $(x_0,y_0,t_0)$,
\[
	\left(\underline \phi_\delta^{-1} \circ \psi\right)_t
		+ \e u \cdot \nabla \left(\underline \phi_\delta^{-1} \circ \psi\right)
		+ |\nabla \left(\underline \phi_\delta^{-1} \circ \psi\right)|
		\leq 0.
\]
Using only the chain rule and the fact that $\underline \phi_\delta' > 0$, we observe that, at $(x_0,y_0,t_0)$,
\[
	\psi_t + \e u \cdot \nabla \psi + |\nabla \psi| \leq 0.
\]

\medskip

Next, we claim that $\underline \mu_\delta \leq G^\e$ on $\R^n \times \{0\}$.  Indeed, we fix any $(x,y) \in \R^n$.  Since the proofs for $y \geq 0$ and $y<0$ are handled similarly, we concentrate on the former case.  If $y \geq \delta$, then
\[
	\underline\mu_\delta(x,y)
		= \underline \phi_\delta( G^\e_\ptw(x,y)) - C_\phi \delta
		\leq \underline \phi_\delta(G^\e_\ptw(x,y))
		\leq \underline \phi_\delta\left(\frac{3y}{2}\right)
		= \underline G\left(\frac{3y}{4}\right)
		\leq \underline G(y)
		\leq G^\e(x,y,0).
\]
The first inequality follows from the fact that $C_\phi \delta \geq 0$.  The second is due to~\eqref{eq:bounds_rho1} and that $\underline \phi_\delta'>0$.  That $\underline G$ is increasing yields the third, while the last is due to \Cref{assumption:initial_datum}.  On the other hand, if $y \in [0,\delta)$,
\[
	\underline \mu_\delta(x,y)
		\leq \underline \phi_\delta\left( \frac{3y}{2}\right) - C_\phi \delta
		\leq \|\underline \phi_\delta\|_{C^{0,1}(-1,1)}\left( \frac{3y}{2}+\delta\right) - C_\phi \delta
		\leq 0
		\leq G^\e(x,y,0).
\]
The first inequality again uses~\eqref{eq:bounds_rho1} and the fact that $\underline \phi_\delta'>0$.  The second is a consequence of the definition of the Lipschitz norm and the fact that $\phi_\delta$ must take the value $0$ somewhere in $(-\delta,\delta)$.   That $y < \delta$ and $C_\phi \geq 4 \|\underline \phi_\delta\|_{C^{0,1}(-1,1)}$ yields the third inequality, while the last follows from \Cref{assumption:initial_datum}.

\medskip

Using the comparison principle and that $\underline \mu_\delta$ is a sub-solution of~\eqref{eq:geometric}, we get that $\underline\mu_\delta \leq G^\e$ in $\R^n \times \R_+$.  After letting $\delta \to 0$, we find
\begin{equation}\label{eq:c111}
	\underline \mu \leq G^\e,
\end{equation}
and, hence,
\begin{equation}\label{eq:sub_level1}
	\{ G^\e \leq 0\}
		\subset \{\underline \mu \leq 0\}
		= \{G_\ptw^\e \leq 0\}. 
\end{equation}

%
%
%

\medskip

A similar argument shows that $\overline \mu := \overline\phi \circ G_\ptw^\e \geq G^\e$, where
\[
	\overline \phi(y) :=
		\begin{cases}
		\overline G(2y), \quad &\text{ if } ~~ y \geq 0,\\
		\overline G(y/2), &\text{ if } ~~ y < 0,
		\end{cases}
\]
and, hence,
\begin{equation}\label{eq:sub_level2}
	\{ G^\e \leq 0\}
		\supset \{\overline \mu \leq 0\}
		= \{G_\ptw^\e\leq 0\}.
\end{equation}
Combining~\eqref{eq:sub_level1} and~\eqref{eq:sub_level2} yields $\{G^\e \leq 0 \} = \{G^\e_\ptw \leq 0\}$.

\medskip

Moreoever, since $\underline \mu \leq G^\e \leq \overline \mu$ and, for all $t\in \R_+$, $\Gamma_t(\underline\mu) = \Gamma_t(\overline \mu) = \Gamma_t(G^\e_\ptw)$, we find $\Gamma_t(G^\e) = \Gamma_t(G^\e_\ptw)$.
\end{proof}

\section{Front asymptotics for the initial value problem of the eikonal equation}\label{sec:FKPP_ivp}


We now obtain estimates on the front location in the general case.  We do so through a simple comparison principle-based argument.

\begin{proof}[Proof of \Cref{prop:ivp}]
	The first inclusion follows from comparison and \Cref{prop:geometric_td}.  Indeed, let $G^\e$ be the solution of~\eqref{eq:geometric} with initial datum $v_0$.   \Cref{prop:geometric_td} gives that $\{ G^\e \leq 0\} = \{ G^\e_\ptw \leq 0\}$.
	
	\medskip
	
	We claim that $G^\e$ is a super-solution of~\eqref{eq:FKPP}.  Fix any test function $\psi$ and suppose that $G^\e - \psi$ has a minimum at $(x,y,t) \in \R^n \times \R_+$.  Then~\eqref{eq:geometric} yields that, at $(x,y,t)$,
	\[
		\psi_t + \e u \cdot \nabla \psi + |\nabla \psi| \geq 0.
	\]
	Using the Cauchy-Schwarz inequality and Young's inequality, at $(x,y,t)$,
	\[
		\psi_t + \e u\cdot \nabla \psi + \frac{1}{2} |\nabla \psi|^2 + \frac{1}{2} \geq 0,
	\]
	and, thus $G^\e$ is a super-solution of~\eqref{eq:FKPP}.
	
	\medskip
	
	Applying the comparison principle, we get that $v^\e \leq G^\e$.  This, in turn, implies that $\{G^\e \leq 0\} \subset \{v^\e \leq 0\}$.  Using the equality above, we obtain $\{G^\e_\ptw \leq 0\} \subset \{v^\e \leq 0\}$.
	
	\medskip
	
	The second inclusion in \Cref{prop:ivp} is a simple case of the maximum principle.  Indeed, $v_0 \geq v^\e_\ptw(\cdot,0)$ in $\R^n$ and $v^\e$ and $v^\e_\ptw$ both satisfy the same equation on $\R^n\times\R_+$.  Hence, $v^\e_\ptw \leq v^\e$ in $\R^n\times\R_+$, from which it follows that $\{v^\e \leq 0\} \subset \{v^\e_\ptw \leq 0\}$, and the proof is complete.
	
%
%
%
%
\end{proof}

\section{Well-posedness and a priori bounds of \eqref{eq:general_metric_problem}} \label{sec:rho}

There are two steps in the proof of \Cref{lem:rho}.  The first is about the existence and uniqueness and some weak bounds on $\rho^\e$.  In the second, which deals with the main difficulty, we bootstrap these weak bounds into sharper, more useful ones.

\medskip

Since $\e$ plays a somewhat reduced role here, for simplicity, we suppress it and write $\rho$ in place of $\rho^\e$.  In addition, since we do not work with time dependence throughout this section we drop the $\aut$ notation and refer to $u_\aut$ as $u$.

\begin{lemma}\label{lem:wp_weak_bounds}
	Suppose \Cref{assumption:advection} holds.  Then there exists a unique globally Lipschitz solution $\rho$ of~\eqref{eq:general_metric_problem} such that, uniformly for all $x \in \R^{n-1}$, $\liminf_{y\to\infty} \rho(x,y) \geq 0$ and $\limsup_{y\to-\infty} \rho(x,y) \leq 0$.  Moreover, there exists $C_L$, depending only on $u$, such that, for all $(x,y)\in \R^n$,
	\begin{equation}\label{eq:weak_bounds_rho}
		|\rho(x,y) - y|
			\leq 3\e \|u\|_\infty |y|
			\qquad \text{ and } \qquad
			\Lip(\rho) \leq C_L.
	\end{equation}
\end{lemma}

%
%

To use the half-relaxed limits, it is necessary to improve~\eqref{eq:weak_bounds_rho}.  This requires to introduce a correction in~\eqref{eq:weak_bounds_rho} that takes care of the oscillations, allowing to construct improved barriers. 

\begin{lemma}\label{lem:sharp_bounds}
	Let $\rho$ be the solution of~\eqref{eq:general_metric_problem} constructed in \Cref{lem:wp_weak_bounds}.  Then there exists positive constants $\mu_1$, $\mu_2$, and $\mu_3$, depending only on $\|u\|_{C^1}$, such that the solution $\rho$ of \eqref{eq:general_metric_problem} satisfies, for all $(x,y) \in \R^n$,
	\[
		|\rho(x,y) - y + \e^{2/3} u_{\parallel} W^\e(\e^{2/3}y)|
			\leq \e^{4/3} \mu_1 |y| + \frac{\mu_2 \e^2 y^2}{2}
				+ \e^{2/3} \mu_3 \bigg|\int_0^{y \e^{2/3}} |W^\e(y')|^2 dy'\bigg|.
	\]
\end{lemma}

It is clear that \Cref{lem:rho} follows directly from \Cref{lem:wp_weak_bounds,lem:sharp_bounds}.  As such, we now aim to prove these two results in turn.

\subsection{Well-posedness and weak bounds}

\begin{proof}[Proof of \Cref{lem:wp_weak_bounds}]
	We proceed in three steps.  Firstly, we establish the existence and uniqueness of solutions of
	\begin{equation}\label{eq:general_metric_problem6}
	\begin{cases}
		- r \frac{\e^\beta}{2} \Delta \rho + r \e u \cdot D \rho + |D \rho|^r = 1
			\qquad &\text{ in }~~ \R^{n-1}\times (\R_- \cup \R_+),\\
		\rho = 0 &\text{ on }~~ \R^{n-1}\times\{0\}.
	\end{cases}
	\end{equation}
	Secondly, we obtain weak bounds on solutions $\rho$ of~\eqref{eq:general_metric_problem6}.  Finally, we use these weak bounds to show that solutions of~\eqref{eq:general_metric_problem6} are solutions of~\eqref{eq:general_metric_problem}; that is, they are solutions on $\R^n$ instead of merely on $\R^{n-1}\times(\R_-\cup\R_+)$.
	
	\bigskip
	
	{\em Step 1:} The existence, uniqueness, and the bound on the Lipschitz constant $C_L$ on $\R^{n-1}\times\R_+$ follows immediately from~\cite[Theorem~A.6]{ArmstrongCardaliaguet}.  A symmetric argument applies on $\R^{n-1}\times \R_-$.
	
	\bigskip
	
	{\em Step 2:} To obtain~\eqref{eq:weak_bounds_rho}, let, for $(x,y) \in \R^{n-1}\times \R_+$, $\underline \rho(x,y) = (1 - 3\e \|u\|_\infty) y$.  It is immediate that
\[
	 1
	 	\geq  r \e \|u\|_\infty (1- 3\e \|u\|_\infty) + (1- 3\e \|u\|_\infty)
	 	\geq - r \frac{\e^\beta}{2} \Delta \underline \rho + r \e u \cdot \nabla \underline \rho + |\nabla \underline\rho|^r.
\]
%

Observe that $\underline \rho \leq \rho$ on $\R^{n-1}\times\{0\}$.  
The comparison principle (see \cite[Proposition~A.4]{ArmstrongCardaliaguet}) yields $\underline \rho \leq \rho$.  

\medskip

We may similarly build a super-solution of~\eqref{eq:general_metric_problem6} on $\R^{n-1}\times \R_+$  
%
%
and conclude that, in $\R^{n-1}\times\R_+$,
\begin{equation}\label{eq:weak_asymptotics}
	(1-3 \e \|u\|_\infty) y
		\leq \rho^\delta
		\leq (1 + 3 \e \|u\|_\infty) y.
\end{equation}

	\bigskip

	{\em Step 3:} We now show that $\rho$ satisfies the planar metric problem~\eqref{eq:general_metric_problem} on $\R^{n-1}\times \{0\}$.  To accomplish this, we look separately at the cases $\beta = \infty$ and $\beta < \infty$.  For simplicity, we show the argument only for $r=1$.  The modifications for the general case are conceptually straightforward but significantly messier.

\medskip

When $\beta = \infty$, we show that, in the classical sense, $\nabla_x \rho(x,0) = 0$ and $\rho_y(x,0) = (1+ \e u_\parallel(x,0))^{-1}$ for all $x\in \R^{n-1}$.  From these two equalities, it is clear that $\rho$ satisfies~\eqref{eq:general_metric_problem} classically on $\R^{n-1}\times\{0\}$.

\medskip

That $\nabla_x \rho \equiv 0$ is obvious since $\rho \equiv 0$ on $\R^{n-1}\times \{0\}$.  We thus focus on proving that $\rho_y(x,0) = (1+ \e u_{\parallel})^{-1}$ for $x\in \R^{n-1}$ by constructing barriers.

\medskip

We begin with a lower bound in $\rho$ for $0 < y \ll 1$. Fix $\delta \in (0, 1/100)$ and let
\[
	\underline \rho = y (1+ u_\parallel w)^{-1} - y^2/(2\delta).
\]
We show that $\underline \rho \leq \rho$ on the domain $V_\delta = \{(x,y) \in \R^{n-1}\times\R_+: y < \delta\}$ by showing that $\underline \rho$ is a sub-solution of~\eqref{eq:general_metric_problem6} on $V_\delta$ and that $\underline \rho \leq \rho$ on $\partial V_\delta$.

\medskip

A direct computation yields
\[\begin{split}
	\e &u \cdot \nabla \underline \rho + |\nabla \underline\rho|
		= - \frac{\e^2 y u_{ \perp} \cdot \nabla_x u_{\parallel} w}{(1+\e u_{\parallel} w)^2} + \e u_{\parallel} w \left(\frac{1}{1 + \e u_{\parallel} w} - \frac{\e y u_{\parallel} w_y}{(1 + \e u_{\parallel} w)^2} - \frac{y}{\delta}\right)\\
		&\quad \left|\frac{\e y \nabla_x u_{\parallel} w}{(1+\e u_{\parallel} w)^2}, \frac{1}{1 + \e u_{\parallel} w} - \frac{\e y u_{\parallel} w_y}{(1 + \e u_{\parallel} w)^2} - \frac{y}{\delta}\right|.
\end{split}\]
Recall that $\e \|u\|_{C^1}, \delta \leq 1/100$ and $0 < y < \delta$.  It then follows from the triangle inequality that
\[\begin{split}
	\e u \cdot \nabla \underline \rho + |\nabla \rho|
		&\leq - \frac{\e^2 y u_{ \perp} \cdot \nabla_x u_{\parallel} w}{(1+\e u_{\parallel} w)^2}
			+ \e u_{\parallel} w \left(\frac{1}{1 + \e u_{\parallel} w}
			- \frac{\e y u_{\parallel} w_y}{(1 + \e u_{\parallel} w)^2} - \frac{y}{\delta}\right)\\
		&\quad + \frac{\e y |\nabla_x u_{\parallel} w|}{(1+\e u_{\parallel} w)^2} + \frac{1}{1 + \e u_{\parallel} w} + \frac{\e y |u_{\parallel} w_y|}{(1 + \e u_{\parallel} w)^2} - \frac{y}{\delta}.
\end{split}\]
Estimating each term in turn and using that $\delta < 1$, we find
\[\begin{split}
	\e u \cdot \nabla \underline \rho + |\nabla \rho|
		&\leq \frac{y}{99^2}
			+ \frac{\e u_{\parallel} w}{1 + \e u_{\parallel} w}
			+ \frac{y}{99^2}
			+ \frac{y}{100 \delta}
			+ \frac{100 y}{99^2}
			+ \frac{1}{1 + \e u_{\parallel} w}
			+ \frac{100 y}{99^2}
			- \frac{y}{\delta}\\
		&\leq 1 + \frac{y}{50} - \frac{y}{50\delta}
		< 0,
\end{split}\]
that is, $\underline \rho$ is a sub-solution of~\eqref{eq:general_metric_problem6} on $V_\delta$.

\medskip

We now show that $\underline \rho \leq \rho$ on $\partial V_\delta$.  Since this is clearly true when $y= 0$, we need only consider the case $y = \delta$.
%
%
%
%
%
For all $x\in \R^{n-1}$, we have
\[
	\underline \rho(x,\delta)
		\leq \frac{\delta}{(1 - 1/100)} - \frac{\delta}{2}
		< \frac{9 \delta}{10},
\]
and, from~\eqref{eq:bounds_rho1},
\[
	\rho(x,\delta)
		\geq \frac{9 \delta}{10}.
\]
It follows that $\underline \rho \leq \rho$ on $\partial V_\delta$.  From the comparison principle, we conclude that $\underline \rho \leq \rho$ in $V_\delta$.

\medskip

A similar argument can be used to conclude that, for $\delta$ sufficiently small, $\rho \leq \overline \rho$ where $\overline \rho(y) := y(1+ \e u_{\parallel}w)^{-1} + y^2/(2\delta)$.

\medskip

We conclude that
\[
	\lim_{y\searrow 0} \frac{\rho(x,y)}{y} = \frac{1}{1 + \e u_{\parallel}w},
\]
and remark that the case when $y \nearrow 0$ follows similarly.    Thus, for all $x\in \R^{n-1}$, $\rho_y(x,0) = (1 + \e u_{\parallel}(x) w(0))^{-1}$, and the proof is complete when $\beta = \infty$.

\medskip

When $\beta < \infty$, the problem is elliptic and the classic theory implies that $\rho \in C^2(\R^n)$ and, hence, that it satisfies~\eqref{eq:general_metric_problem}.  This concludes the proof.

\end{proof}

\subsection{Sharper a priori estimates}

We now show how to bootstrap the weak bounds obtained above to the sharp bounds on $\rho$ necessary to control the corrector $\chi^\e_\aut$ defined in~\eqref{eq:chi_definition}.

\begin{proof}[Proof of \Cref{lem:sharp_bounds}]
Firstly we notice that we need only obtain bounds for all $\epsilon \in (0,\epsilon_0)$ for some threshold $\epsilon_0 >0$, to be determined.  For $\epsilon \geq \epsilon_0$ this is trivially true by \Cref{lem:wp_weak_bounds} after taking $\mu_1$, $\mu_2$, and $\mu_3$ sufficiently large.  Secondly, we work only on $\R^{n-1}\times \R_+$, since the case $y< 0$ can be handled similarly.

%

\bigskip

{\em Step 1:}  
To obtain a lower bound, we build a sub-solution.  Fix positive constants $\mu_1$, $\mu_2$, and $\mu_3$ to be determined, and let
\[
	\underline \rho(x,y)
		:= y(1 - \e^{4/3}\mu_1) - \frac{1}{2} \mu_2 \e^2 y^2 - \mu_3 \e^{2/3}\int_0^{y\e^{2/3}} |W^\e(y')|^2 dy' - \e^{2/3} u_{\parallel}(x) W^\e(\e^{2/3}y).
\]
Direct computations yield
\[\begin{split}
	- r&\frac{\e^\beta}{2} \Delta \underline \rho
		+ r\e u \cdot \nabla \underline \rho
		+ |\nabla \underline\rho|^r\\
	&= r\frac{\e^\beta}{2} \left(\mu_2 \e^2 + 2\mu_3 \e^{5/3} W^\e(\e^{2/3} y) w(y) + \e u_{\parallel} w_y + \e^{2/3}\Delta_x u_{\parallel}(x) W^\e(\e^{2/3}y) \right)\\
	&\qquad - r \e^{5/3} W^\e(\e^{2/3}y) u_{\perp} \cdot \nabla_x u_{\parallel}
		+ r \e u_{\parallel} w \left(1 - \mu_1 \e^{4/3} - \mu_2 \e^2 y - \mu_3 \e^{4/3} |W^\e(\e^{2/3}y)|^2 - \e u_{\parallel} w\right)\\
	&\qquad +\Big[\e^{4/3} |\nabla_x u_{\parallel}|^2|W^\e(\e^{2/3}y)|^2 + 1 - 2 \left(\mu_1 \e^{4/3} + \mu_2 \e^2 y + \mu_3 \e^{4/3} |W^\e(\e^{2/3}y)|^2 + \e u_{\parallel} w\right)\\
	&\qquad	+ \left(\mu_1 \e^{4/3} + \mu_2 \e^2 y + \mu_3 \e^{4/3} |W^\e(\e^{2/3}y)|^2 + \e u_{\parallel} w\right)^2 \Big]^{r/2}.
\end{split}\]
After using the inequality $(1+x)^{r/2} \leq 1 + rx/2$ and cancelling two terms of the form $\e u_{\parallel} w$, which is the purpose for the last term in $\underline \rho$, we find
\[\begin{split}
	- r&\frac{\e^\beta}{2} \Delta \underline \rho
		+ r\e u \cdot \nabla \underline \rho
		+ |\nabla \underline\rho|^r\\
	&\leq r\frac{\e^\beta}{2} \left(\mu_2 \e^2 + 2\mu_3 \e^{5/3} W^\e(\e^{2/3}y) w(y) + \e u_{\parallel} w_y + \e^{2/3}\Delta_x u_{\parallel}(x) W^\e(\e^{2/3}y) \right)\\
	&\qquad - r \e^{5/3} W^\e(\e^{2/3}y) u_{\perp} \cdot \nabla_x u_{\parallel}\\
	&\qquad	- r \e u_{\parallel} w \left( \mu_1 \e^{4/3} + \mu_2 \e^2 y + \mu_3 \e^{4/3} |W^\e(\e^{2/3}y)|^2 + \e u_{\parallel} w\right)
		+ 1\\
	&\qquad + \frac{r}{2}\e^{4/3} |\nabla_x u_{\parallel}|^2 |W^\e(\e^{2/3}y)|^2
		 - r \left(\mu_1 \e^{4/3} + \mu_2 \e^2 y + \mu_3 \e^{4/3} |W^\e(\e^{2/3}y)|^2\right)\\
	&\qquad	+ \frac{r}{2}\left(\mu_1 \e^{4/3} + \mu_2 \e^2 y + \mu_3 \e^{4/3} |W^\e(\e^{2/3}y)|^2 + \e u_{\parallel} w\right)^2.
\end{split}\]
Next, we rearrange terms and we use that $(a_1 + \cdots + a_k)^2 \leq k (a_1^2 + \cdots + a_k^2)$ and $r\leq 2$ to obtain, for some $C\geq 1$ depending only on $\|u\|_{C^2}$ and $\|w\|_{C^1}$ and changing line-by-line,
\[\begin{split}
	- r&\frac{\e^\beta}{2} \Delta \underline \rho
		+ r\e u \cdot \nabla \underline \rho
		+ |\nabla \underline\rho|^r\\
	&\leq 1
			- \e^{4/3}\Big[r\mu_1 
				- r  \e^{\beta + 2/3} \mu_2
				- r\e^{1/3} u_{\parallel} w_y
				 + r \e  \mu_1 u_{\parallel} w
				 + r\e^{2/3}(u_{\parallel} w)^2\\
&\qquad		 - 2r \mu_1^2 \e^{4/3}
				 - 2r \e^{2/3} (u_{\parallel} w)^2\Big]
			- \e^2 y\Big[r\mu_2
				+ r\e^\beta \mu_3 \frac{W^\e(\e^{2/3}y)}{y \e^{1/3}} w\\
&\qquad	 	+ r \frac{W^\e(\e^{2/3}y)}{y\e^{1/3}} u_{\perp} \cdot \nabla u_{\parallel} + r \mu_2 \e u_{\parallel} w \Big]
			- \e^{4/3} |W^\e(\e^{2/3}y)|^2 \Big[ r\mu_3 + r  \mu_3 \e u_{\parallel} w - \frac{r}{2} |\nabla_x u_{\parallel}|^2\Big]\\
&\qquad	+ \frac{r}{2}  \e^{\beta + 2/3} \Delta_x u_{\parallel} W^\e(\e^{2/3}y)
			+ 2r \mu_2^2 \e^4 y^2
			+ 2r \mu_3^2 \e^{8/3} |W^\e(\e^{2/3}y)|^4\\
		&\leq 1
			- \e^{4/3}\Big[r\mu_1 
				- C\left( \e^{4/3} \mu_2 + \e \mu_1 + \e^{4/3} \mu_1^2  + \e^{1/3}\right)\Big]
			- \e^2 y\Big[r\mu_2 - C\Big(\frac{|W^\e(\e^{2/3}y)|}{y\e^{1/3}} ( \e^{2/3}\mu_3 + 1) + \e \mu_2 \Big) \Big]\\
&\qquad	- \e^{4/3} |W^\e(\e^{2/3}y)|^2 \Big[ r\mu_3 - C(\mu_3 \e +1)\Big]
		+ C\e^{4/3} |W^\e(\e^{2/3}y)|
		+ 4 \mu_2^2 \e^4 y^2
		+ 4 \mu_3^2 \e^{8/3} |W^\e(\e^{2/3}y)|^4.
\end{split}\]
Young's inequality and that $|W^\e(\e^{2/3}y)| \leq C \e^{1/3} y$ yields
\[\begin{split}
	- r&\frac{\e^\beta}{2} \Delta \underline \rho
		+ r\e u \cdot \nabla \underline \rho
		+ |\nabla \underline\rho|^r\\
		&\leq 1
			- \e^{4/3}\Big[r\mu_1 
				- C\left( \e^{4/3} \mu_2 + \mu_1 \e^{2/3} + \mu_1^2 \e^{4/3} +  \e^{1/3}\right)\Big]
			- \e^2 y\Big[r\mu_2 - C\Big( \e^{2/3} \mu_3 + 1 + \e \mu_2 \Big) \Big]\\
&\qquad	- \e^{4/3} |W^\e(\e^{2/3}y)|^2 \Big[ r\mu_3 - C(\mu_3 \e + 1) \Big]
		+ C\e^{4/3} (1 + |W^\e(\e^{2/3}y)|^2)
	+ 4 \e^4 \mu_2^2 y^2
		+ 4 \mu_3^2 \e^{8/3} |W^\e(\e^{2/3}y)|^4.
\end{split}\]
Rearranging terms and, if necessary, lowering $\epsilon_0$ so that $C \epsilon_0^{2/3} < 1/2$, we find
\begin{equation}\label{eq:c22}
\begin{split}
	- r&\frac{\e^\beta}{2} \Delta \underline \rho
		+ r\e u \cdot \nabla \underline \rho
		+ |\nabla \underline\rho|^r\\
		&\leq 1
			- \e^{4/3}\left[\frac{\mu_1}{2} 
				- C(\e^{4/3} \mu_2 + \e \mu_1^2 + 1)\right]
			- \e^2 y\left[\frac{1}{2}\mu_2 - C(1 + \e^{2/3} \mu_3) \right]\\
&\qquad		- \e^{4/3} |W^\e(\e^{2/3}y)|^2 \left[\frac{\mu_3}{2} - C\right]
			+ 4 \e^4 \mu_2^2 y^2
			+ 4 \mu_3^2 \e^{8/3} |W^\e(\e^{2/3}y)|^4.
\end{split}
\end{equation}

Recall, from the definition of mild white noise, that $\|w\|_{C^1} \leq M$, and let
\begin{equation}\label{eq:mu_23}
	\mu_3 := 4C + 1
		\qquad \text{ and } \qquad
	\mu_2 := 4C + 1 + 8 M^2 \sqrt{\mu_3} (1 + \|u\|_\infty).
\end{equation}
Let $\overline \e_0>0$ be such that
\[
	\mu_2 \geq 4 C (1 + \overline\e_0^{2/3} \mu_3),
\]
and set $\mu_1 = 4 C(\overline \e_0^{4/3} \mu_2 + 1)$.

\medskip

Lowering $\e_0$, if necessary, so that $\e_0 \leq \overline \e_0$, we find
\[\begin{split}
	- r&\frac{\e^\beta}{2} \Delta \underline \rho
		+ r\e u \cdot \nabla \underline \rho
		+ |\nabla \underline\rho|^r\\
		&\leq 1
			- \e^{4/3}\left[\frac{\mu_1}{4} 
				- C\e \mu_1^2\right]
			- \e^2 \frac{\mu_2}{4} y
			- \e^{4/3} \frac{\mu_3}{4} |W^\e(\e^{2/3}y)|^2
			+ 4 \e^4 \mu_1^2 y^2
			+ 4 \mu_3^2 \e^{8/3} |W^\e(\e^{2/3}y)|^4.
\end{split}\]
Again, making $\e_0$ even smaller, if necessary, we obtain $8 C \e_0 \mu_1 \leq 1$ and, hence,
\begin{equation}\label{eq:almost_subsolution}
\begin{split}
	- r&\frac{\e^\beta}{2} \Delta \underline \rho
		+ r\e u \cdot \nabla \underline \rho
		+ |\nabla \underline\rho|^r\\
		&\leq 1
			- \e^{4/3} \frac{\mu_1}{8}
			- \e^2 \frac{\mu_2}{4} y
			- \e^{4/3} \frac{\mu_3}{4} |W^\e(\e^{2/3}y)|^2
			+ 4 \e^4 \mu_1^2 y^2
			+ 4 \mu_3^2 \e^{8/3} |W^\e(\e^{2/3}y)|^4.
\end{split}
\end{equation}

\medskip

We show next that $\underline \rho$ is a sub-solution of~\eqref{eq:general_metric_problem} in the domain $V_\e = \{(x,y) \in \R^{n-1}\times\R_+ : y < (16 \mu_3 M^4 \e^2)^{-1/2}\}$.  
Consider the third and fifth terms in the right hand side of~\eqref{eq:almost_subsolution}.  Making $\e_0$ smaller and using the definition of $V_\e$, we find
\begin{equation}\label{eq:c20}
	4 \e^4 \mu_2^2 y^2 - \frac{\e^2 \mu_2 y}{4}
		= \frac{\e^2 y \mu_2}{4}( 16 \e^2 \mu_2 y - 1)
		< \frac{\e^2 y \mu_2}{4}
			\left( \frac{4 \e \mu_2}{M^2\sqrt{\mu_3}} - 1\right)
		< 0
		\quad \text{ in }~~ V_\e.
\end{equation}
Next,  consider the fourth and six terms in the right hand side of~\eqref{eq:almost_subsolution}.  Since $|W^\e(\e^{2/3}y)|^2 \leq \e^{2/3} M^2 y^2$ and $\mu_2$, $\mu_3$, $M\geq 1$,
\begin{equation}\label{eq:c21}
\begin{split}
	&4  \e^{8/3} \mu_3^2 |W^\e(\e^{2/3}y)|^4
		- \frac{\e^{4/3} \mu_3}{4} |W^\e(\e^{2/3}y)|^2\\
		&= \frac{\e^{4/3} \mu_3}{4} |W^\e(\e^{2/3}y)|^2
			\left(16 \e^{4/3} \mu_3 |W^\e(\e^{2/3}y)|^2 - 1 \right)\\
		&\leq \frac{\e^{4/3} \mu_3}{4} |W^\e(\e^{2/3}y)|^2
			\left(16 \e^2 \mu_3 M^2 y^2 - 1 \right)
		< \frac{\e^{4/3} \mu_3}{4} |W^\e(\e^{2/3}y)|^2
			\left(\frac{1}{M^2} - 1 \right)
		\leq 0
		\quad \text{ in } V_\e.
\end{split}
\end{equation}
The combination of~\eqref{eq:almost_subsolution}, \eqref{eq:c20}, and~\eqref{eq:c21} imply that $\underline\rho$ is a sub-solution of~\eqref{eq:general_metric_problem} on $V_\e$.

\medskip

Next, we claim that $\underline\rho \leq \rho$ on $\partial V_\e$. Since clearly $\underline\rho \leq \rho$ on $\R^{n-1}\times \{0\}$, we concentrate on $\R^{n-1}\times \{(16 \mu_3 M^4 \e^2)^{-1/2}\}$.  Using the weak lower bound of \Cref{lem:wp_weak_bounds} and that $u_{\parallel} W^\e(\e^{2/3} y) \geq - \e^{1/3} \|u\|_\infty y$, we observe that
\[\begin{split}
	\rho(x,y) - \underline\rho(x,y)
		&\geq y(1 - \e C_L) - \underline\rho(x,y) \\
		&= \e^{4/3} \mu_1 y + \frac{1}{2} \mu_2 \e^2 y^2 + \mu_3 \e^{2/3} \int_0^{y\e^{2/3}} |W^\e(y')|^2 dy' + \e^{2/3} u_{\parallel} W^\e(\e^{2/3}y) - \e C_L y\\
		&\geq \e^{4/3} \mu_1 y + \frac{1}{2} \mu_2 \e^2 y^2 + \mu_3 \e^{2/3} \int_0^{y\e^{2/3}} |W^\e(y')|^2 dy' - \e (C_L + \|u\|_\infty ) y.
\end{split}\]
Thus, on $\R^{n-1}\times\{(16 \mu_3 M^4 \e^2)^{-1/2}\}$,
\[\begin{split}
	\rho(x,y) - \underline\rho(x,y)
		&\geq  \e^{4/3} \mu_1 y + \frac{\mu_2 \e^2}{2} \frac{y}{4 \sqrt \mu_3 M^2 \e} + \mu_3 \e^{2/3} \int_0^{y\e^{2/3}} |W^\e(y')|^2 dy' - \e (C_L + \|u\|_\infty)y.
\end{split}\]
The choice of $\mu_2$ and $\mu_3$ (see~\eqref{eq:mu_23}) gives that the sum of the second and fourth terms on the right hand side is positive, and, hence,
\[
	\rho(x,y) - \underline\rho(x,y)
		\geq  \e^{4/3} \mu_1 y + \mu_3 \e^{2/3} \int_0^{y\e^{2/3}} |W^\e(y')|^2 dy
		\geq 0.
\]
It then follows from the comparison principle that $\underline \rho \leq \rho$ on $V_\e$.

\medskip

A similar argument 
shows that $\underline \rho \leq \rho$ for $y > (16 \mu_3 M^4 \e^2)^{-1/2}$, so we omit the details.  We conclude that $\underline \rho \leq \rho$ in $\R^{n-1}\times\R_+$, finishing the proof of the lower bound.

\bigskip

{\em Step 2::}
We obtain an upper bound on $\rho$ by constructing a super-solution and arguing as above.  As such, we only include the first steps, which vary from those of the proof of the lower bound.  The rest of the proof proceeds exactly as above.

\medskip

Fix positive constants $\mu_1$, $\mu_2$, and $\mu_3$ to be determined and let
\[
	\overline \rho(x,y)
		:= y(1+\e^{4/3} \mu_1)
			+ \frac{1}{2} \mu_2 \e^2 y^2
			+ \e^{2/3} \mu_3\int_0^{y\e^{2/3}} |W^\e(y')|^2dy'
			- \e^{2/3} u_{\parallel} W^\e(\e^{2/3}y).
\]
A direct computation gives
\begin{equation}\label{eq:c101}
\begin{split}
	- r&\frac{\e^\beta}{2} \Delta \underline \rho
		+ r\e u \cdot \nabla \underline \rho
		+ |\nabla \underline\rho|^r\\
	&= - r\frac{\e^\beta}{2}  \left(\mu_2 \e^2 + 2\mu_3 \e^{5/3} W^\e(\e^{2/3}y) w(y) - \e u_{\parallel} w_y - \e^{2/3}\Delta_x u_{\parallel}(x) W^\e(\e^{2/3}y) \right)\\
		&\qquad - r \e^{5/3} W^\e(\e^{2/3}y) u_{\perp} \cdot \nabla_x u_{\parallel}
			+ r \e u_{\parallel} w \left(1 + \mu_1 \e^{4/3} + \mu_2 \e^2 y + \mu_3 \e^{4/3} |W^\e(\e^{2/3}y)|^2 - \e u_{\parallel} w\right)\\
		&\qquad + \Big[\e^{4/3} |\nabla_x u_{\parallel}|^2 + 1 + 2 \left(\mu_1 \e^{4/3} + \mu_2 \e^2 y + \mu_3 \e^{4/3} |W^\e(\e^{2/3}y)|^2 - \e u_{\parallel} w\right)\\
		&\qquad	+ \left(\mu_1 \e^{4/3} + \mu_2 \e^2 y + \mu_3 \e^{4/3} |W^\e(\e^{2/3}y)|^2 - \e u_{\parallel} w\right)^2 \Big]^{r/2}\\
	&\geq - r\frac{\e^\beta}{2} \left(\mu_2 \e^2 + 2\mu_3 \e^{5/3} W^\e(\e^{2/3}y) w(y) - \e u_{\parallel} w_y - \e^{2/3}\Delta_x u_{\parallel}(x) W^\e(\e^{2/3}y) \right)\\
		&\qquad - r \e^{5/3} W^\e(\e^{2/3}y) u_{\perp} \cdot \nabla_x u_{\parallel}
		+ r \e u_{\parallel} w \left(1 + \mu_1 \e^{4/3} + \mu_2 \e^2 y + \mu_3 \e^{4/3} |W^\e(\e^{2/3}y)|^2 - \e u_{\parallel} w\right)\\
		&\qquad +\Big[\e^{4/3} |\nabla_x u_{\parallel}|^2 |W^\e(\e^{2/3}y)|^2 + 1 + 2 \left(\mu_1 \e^{4/3} + \mu_2 \e^2 y + \mu_3 \e^{4/3} |W^\e(\e^{2/3}y)|^2 - \e u_{\parallel} w\right) \Big]^{r/2}.
\end{split}
\end{equation}
In the proof of the lower bound, we used the concavity of $(1 + x)^{r/2}$; this will not work here.  Instead, we use Taylor's theorem, which implies that there exists $E_\e$ such that
\[
	|E_\e|
			\leq \left| 2\left(\mu_1 \e^{4/3} + \mu_2 \e^2 y + \mu_3 \e^{4/3} |W^\e(\e^{2/3}y)|^2 - \e u_{\parallel} w\right)
				 + \e^{4/3}|\nabla_x u_{\parallel}|^2 |W^\e(\e^{2/3}y)|^2\right|
\]
and
\[\begin{split}
	&\Big[\e^{4/3} |\nabla_x u_{\parallel}|^2 + 1 + 2 \left(\mu_1 \e^{4/3} + \mu_2 \e^2 y + \mu_3 \e^{4/3} |W^\e(\e^{2/3}y)|^2 - \e u_{\parallel} w\right) \Big]^{r/2}\\
&\quad	= 1
				+ r\left(\mu_1 \e^{4/3} + \mu_2 \e^2 y + \mu_3 \e^{4/3} |W^\e(\e^{2/3}y)|^2 - \e u_{\parallel} w\right)
				 + \frac{r}{2}\e^{4/3}|\nabla_x u_{\parallel}|^2|W^\e(\e^{2/3}y)|^2\\
&\qquad		- \frac{r(2-r)}{4(1 + E_\e)^{3/2}}
					\left(2\left(\mu_1 \e^{4/3} + \mu_2 \e^2 y + \mu_3 \e^{4/3} |W^\e(\e^{2/3}y)|^2 - \e u_{\parallel} w\right)
				 + \e^{4/3}|\nabla_x u_{\parallel}|^2|W^\e(\e^{2/3}y)|^2 \right)^2.
\end{split}\]
In view of $\e \|u\|_\infty \leq 1/4$, we find $|E_\e| \leq 1/2$.  Using this with the identity above, we find
\[\begin{split}
	&\Big[\e^{4/3} |\nabla_x u_{\parallel}|^2 + 1 + 2 \left(\mu_1 \e^{4/3} + \mu_2 \e^2 y + \mu_3 \e^{4/3} |W^\e(\e^{2/3}y)|^2 - \e u_{\parallel} w\right) \Big]^{r/2}\\
&\quad	 \geq 1
				+ r\left(\mu_1 \e^{4/3} + \mu_2 \e^2 y + \mu_3 \e^{4/3} |W^\e(\e^{2/3}y)|^2 - \e u_{\parallel} w\right)
				 + \frac{r}{2}\e^{4/3}|\nabla_x u_{\parallel}|^2|W^\e(\e^{2/3}y)|^2\\
&\qquad		- \frac{r(2-r)}{4(1/2)^{3/2}}
					\left(2\left(\mu_1 \e^{4/3} + \mu_2 \e^2 y + \mu_3 \e^{4/3} |W^\e(\e^{2/3}y)|^2 - \e u_{\parallel} w\right)
				 + \e^{4/3}|\nabla_x u_{\parallel}|^2|W^\e(\e^{2/3}y)|^2 \right)^2.
\end{split}\]
Inserting the last estimate into~\eqref{eq:c101} and using that $4 \cdot 2^{-3/2} \geq 1$, we find
\[\begin{split}
	- r&\frac{\e^\beta}{2} \Delta \underline \rho
		+ r\e u \cdot \nabla \underline \rho
		+ |\nabla \underline\rho|^r\\
	&\geq - r\frac{\e^\beta}{2} \left(\mu_2 \e^2 + 2\mu_3 \e^{5/3} W^\e(\e^{2/3}y) w(y) - \e u_{\parallel} w_y - \e^{2/3}\Delta_x u_{\parallel}(x) W^\e(\e^{2/3}y) \right)\\
	&\quad - r \e^{5/3} W^\e(\e^{2/3}y) u_{\perp} \cdot \nabla_x u_{\parallel}
		+ r \e u_{\parallel} w \left(1 + \mu_1 \e^{4/3} + \mu_2 \e^2 y + \mu_3 \e^{4/3} |W^\e(\e^{2/3}y)|^2 - \e u_{\parallel} w\right)\\
	&\quad +1
				+ r\left(\mu_1 \e^{4/3} + \mu_2 \e^2 y + \mu_3 \e^{4/3} |W^\e(\e^{2/3}y)|^2 - \e u_{\parallel} w\right)
				 + \frac{r}{2}\e^{4/3}|\nabla_x u_{\parallel}|^2\\
&\quad		- r(2-r)
					\left(2\left(\mu_1 \e^{4/3} + \mu_2 \e^2 y + \mu_3 \e^{4/3} |W^\e(\e^{2/3}y)|^2 - \e u_{\parallel} w\right)
				 + \e^{4/3}|\nabla_x u_{\parallel}|^2|W^\e(\e^{2/3}y)|^2 \right)^2.
\end{split}\]
As before, after rearranging terms, applying Young's inequality,  bounding terms involving $u$, and using the inequality $(a_1 + \cdots + a_k)^2 \leq k (a_1^2 + \cdots + a_k^2)$, we get, for some $C\geq 1$ depending only on $\|u\|_{C^2}$ and $\|w\|_{C^1}$,
\begin{equation}\label{eq:c23}
\begin{split}
	- r&\frac{\e^\beta}{2}  \Delta \underline \rho
		+ r\e u \cdot \nabla \underline \rho
		+ |\nabla \underline\rho|^r\\
	&\geq 1
		+ \e^{4/3} \Big[ r \mu_1
				- C(\e^{4/3} \mu_2 + 1  + \e \mu_1 + \e^{4/3} \mu_1^2)
			\Big]
		+ \e^2 y \Big[ r\mu_2
				- C(\e^{2/3}\mu_3
				+ 1
				+ \e \mu_2)
			\Big]\\
&\quad
		+ \e^{4/3} |W^\e(\e^{2/3}y)|^2 \Big[ r \mu_3
				- C(\e \mu_3 +1)
			\Big]
		- C \e^4 \mu_2^2 y^2
		- C \e^{8/3} (\mu_3^2 + 1) |W^\e(\e^{2/3}y)|^4.
\end{split}
\end{equation}

At this point, we notice that~\eqref{eq:c23} is analogous to~\eqref{eq:c22} in the proof of the lower bound.  As the rest of the proof proceeds in the exact same manner, we omit it.  We conclude that $\overline \rho \geq \rho$ in $\R^{n-1}\times\R_+$, finishing the proof.
\end{proof}

\bibliographystyle{abbrv}
\bibliography{refs}
\end{document}